\documentclass[12pt]{article}

\usepackage{latexsym}

\usepackage{graphicx}
\usepackage{amsmath}
\usepackage{dsfont}
\usepackage{amssymb}
\usepackage{amsthm}

\newtheorem{theorem}{Theorem}[section]

\newtheorem{lemma}[theorem]{Lemma}

\def\E{{\mathds E}}

\def\P{{\mathds P}}
\def\R{{\mathds{R}}}

\def\I{{\mathds 1}}

\def\cE{{\mathcal E}}
\def\cF{{\mathcal F}}

\def\cJ{{\mathcal J}}

\def\cN{{\mathcal N}}
\def\cS{{\mathcal S}}

\def\a{{\alpha}}
\def\b{{\beta}}
\def\e{{\varepsilon}}

\def\l{{\lambda}}

\def\t{{\tau}}

\def\bd{{\partial }}

\DeclareMathOperator{\aff}{aff}
\DeclareMathOperator{\lin}{lin}

\DeclareMathOperator{\diam}{diam}

\def\bm #1{{\boldsymbol{#1}}}

\begin{document}

\title{The convex hull of random points \\ on the boundary of a simple polytope}
\author{M. Reitzner, C. Sch\"utt, E. M. Werner\thanks{Partially supported by NSF grant DMS 1811146}}
\maketitle

\begin{abstract}
The convex hull of $N$ independent random points chosen on the boundary of a simple polytope in $ \R^n$ is investigated. 
Asymptotic formulas for the expected number of vertices and facets, and for the expectation of the volume difference are derived. 
This is one of the first investigations leading to rigorous results for random polytopes which are neither simple nor simplicial.
The results contrast existing results when points are chosen in the interior of a convex set.
\end{abstract}

\section{Introduction and statement of results}
Let $K$ be a convex set of dimension $n$ in $\R^n$, $n \geq 2$. Let $N \in \mathbb{N}$ and choose $N$ random points $X_1, \dots, X_N$  uniformly distributed either in the interior of $K$ or on the boundary $\bd K$ of $K$, and denote by $P_N= [X_1, \dots, X_N]$ the convex hull of these points. The expected number of vertices $\E f_0(P_N)$, the expected number of $(n-1)$-dimensional faces $\E f_{n-1}(P_N)$, and the expectation of the volume difference $V_n(K)-\E V_n (P_N)$ of $K$ and $P_N$ are of interest. Since explicit results for fixed $N$ cannot be expected one investigates the asymptotics as $N \to \infty$.

If the vertices of the random polytopes are chosen \emph{from the interior} of a convex set, there is a vast amount of literature. Investigations started with two famous articles by R\'enyi and Sulanke \cite{RS1}, \cite{RS2} who obtained in the planar case the asymptotic behavior of the expected area $\E V_2(P_N)$ when the boundary of $K$ is sufficiently smooth and when $K$ is a polygon.
In a series of papers these formulae were generalized to higher dimensions. In the case when the boundary of $K$ is sufficiently smooth, we know by work of Wieacker \cite{Wi}, Schneider and Wieacker \cite{SchWi1}, B\'ar\'any 
\cite{Bar2}, Sch\"utt \cite{Schu2} and B\"or\"oczky, Hoffmann and Hug \cite{BHH} that the volume difference behaves like
\begin{equation}\label{eq:EVi}
V_n(K)-\E V_n (P_N) = c_n \Omega(K) \, V_n(K)^\frac 2 {n+1} N^{- \frac 2 {n+1}} \left(1+o(1)\right),
\end{equation}
where $P_N$ is the convex hull of uniform iid random points in the interior of $K$, $\Omega(K)$ denotes the affine surface area of $K$ and  $c_n$ is a constant that depends only on $n$.
The generalization to all intrinsic volumes is due to B\'ar\'any \cite{Bar2},  \cite{Bar2c} and Reitzner \cite{Re3}. 
The corresponding results for random points chosen in a polytope $P$ are much more difficult. In a long and intricate proof 
B\'ar\'any and Buchta \cite{BB} settled the case of polytopes $P \subset \R^n $,
$$
V_n(P)-\E V_n(P_N) = \frac{ \operatorname{flag}(P)}{(n+1)^{n-1} (n-1)!} N^{-1} (\ln N)^{n-1} (1+o(1)),
$$
where $ \operatorname{flag}(P)$ is the number of flags of the polytope $P$. A {\sl flag} is a sequence of $i$-dimensional faces $F_i$ of $P$, $i=0, \dots , n-1$, such that $F_i \subset F_{i+1}$. The phenomenon that the expression should only depend on this combinatorial structure of the polytope had been discovered in connection with floating bodies by Sch\"utt \cite{Schu1}.

\bigskip
Due to Efron's identity \cite{Ef} the results on $\E V_n (P_N)$ can be used to determine the expected number of vertices of $P_N$. 
The general results for the number of $\ell$-dimensional faces $f_\ell(P_N)$ are due to Wieacker \cite{Wi}, B\'ar\'any and Buchta \cite{BB} and Reitzner \cite{Re7}: if $K$ is a smooth convex body and $\ell \in \{0, \dots, n-1\}$, then 
\begin{equation} \label{eq:C2Efl}
\E f_\ell (P_N) = c_{n,\ell} \, \Omega(K) \, N^{ \frac {n-1} {n+1}}(1+o(1)),
\end{equation} 
and if $P$ is a polytope, then, with a different constant, but still denoted  $c_{n,\ell}$, 
\begin{equation} \label{eq:PdEfl}
\mathbb E f_\ell (P_N) =  c_{n,\ell}\,   \operatorname{flag}(P) \, (\ln N)^{ {n-1}} (1+o(1)) .
\end{equation}
Choosing random points from the interior of a convex body always produces a simplicial polytope with probability one. Yet often applications of the above mentioned results in computational geometry, the analysis of the average complexity of algorithms and optimization necessarily deal with non simplicial polytopes and it became crucial to have analogous results for random polytopes without this very specific combinatorial structure.
The only classical results for this question concern $0/1$-polytopes in high dimensions \cite{BPo, DFMc2, GGM1, GGM2, MPR}, which have a highly interesting combinatorial structure,  yet in a very specific setting. 
And very recently Newman \cite{Newman} used a somewhat dual approach to construct general random polytopes from random polyhedra.

In view of the applications it is also of high interest to show that the face numbers of most realizations of random polytopes are close to the expected ones, and thus to prove variance estimates, central limit theorems and deviation inequalities. There has been serious progress in this direction in recent years, and we refer to the survey articles \cite{surv-DH, surv-DHMR, surv-MR}. 

In all these results there is a general scheme: if the underlying convex sets are smooth then the number of faces and the volume difference behave asymptotically like powers of $N$,  if the underlying sets are convex polytopes then logarithmic factors show up. Metric and combinatorial quantities only differ by a factor $N$.

In this paper we are discussing the case that the random points are chosen from the boundary of a polytope $P$. 
In high dimensions, this produces random polytopes which are neither simple nor 
simplicial. 
We  point out  though that still most of the facets are simplices.
Thus our results are a decisive step in taking into account the point mentioned above. The applications in computational geometry, the analysis of the average complexity of algorithms and optimization need formulae for the combinatorial structure of the involved random polytopes and thus the question on the number of facets and vertices are of interest. 

From (\ref{eq:PdEfl}) it follows immediately that for random polytopes whose points are chosen from the
boundary of a polytope the expected number of vertices is
$$
\E f_{0}(P_{N})
=c_{n-1,0} \operatorname{flag}(P) \, (\ln N)^{ {n-2}} (1+o(1))
$$
with  $c_{n-1,0}$ from (\ref{eq:PdEfl}), independent of $P$.
Indeed, a chosen point is a vertex of a random polytope if and only if it is a vertex of the 
convex hull of all the random points chosen in the same facet of $P$. 
We define $\ln_+ x=\max\{0, \ln x\}$.
By (\ref{eq:PdEfl}) we get that the expected number of vertices 
equals
$$
c_{n-1,0} \sum_{F_i} \operatorname{flag}(F_i)\E (\ln_+ N_{i} )^{ {n-2}} (1+o(1)),
$$
where we sum over all facets $F_i$ of $P$ and $N_{i}$ is a binomial distributed random variable with parameters $N$ and $p_i=\l_{n-1}(F_i)/(\sum_{F_j} \l_{n-1}(F_j))$. Here $\l_{n-1}$ is the $(n-1)$-dimensional Lebesgue measure. It is left to observe that
$ \E (\ln_+ N_{i})^{n-2} = (\ln N)^{n-2}(1+o(1))$.
and
$ \sum_{F_i} \operatorname{flag}(F_i) = \operatorname{flag}(P) $. 

For our first main results we  restrict our investigations to simple polytopes. We recall that a polytope in $\mathbb{R}^n$
is called simple if each of its vertices is contained in exactly $n$ facets.
\begin{theorem}\label{th:EfN}
Let $n \geq 2$ and choose $N$ uniform random points on the boundary of a simple polytope $P$ in $\mathbb{R}^n$, $n \geq 2$.
For the expected number of facets of the random polytope $P_N$, we have
\begin{eqnarray*}
\E f_{n-1}(P_N)  
&=& 
c_{n} f_0( P) (\ln N)^{n-2} (1+O((\ln N)^{-1})),
\end{eqnarray*}
with some $ c_{n} >0$ independent of $P$.
\end{theorem}
The case $n=2$ is particularly simple. $\E f_{1}(P_N)$ is  asymptotically, as $N \to \infty$, equal to $2  f_0( P) = 2 f_1(P)$.

Note that for a simplicial polytope $\operatorname{flag}(P) = n! f_0(P)$ and therefore Theorem \ref{th:EfN} can also be written as 
$$\E f_{n-1}(P_N)  
=
\frac{c_{n}}{n!} \operatorname{flag}(P) (\ln N)^{n-2} (1+O((\ln N)^{-1})) $$
We conjecture this formula to hold for general polytopes.
Yet this seems to be much more involved. We are showing here that for $1 \leq \ell \leq n-2$
$$
\E f_{\ell}(P_N)  
\geq
c_{n-1,\ell} \operatorname{flag}(P) (\ln N)^{n-2} (1+o(1))
$$
with $ c_{n-1,\ell}$ defined in \eqref{eq:PdEfl}.
For this we count those $\ell$-dimensional faces which are contained in the facets $F_i$ of $P$. Analogous to the case $\ell =0$ we have 
\begin{align*}
\E f_{\ell}(P_N)  
& \geq
\sum_{F_i} \E f_{\ell}(P_N \cap F_i)  
\\ & =
c_{n-1,\ell} \sum_{F_i} \operatorname{flag}(F_i) \E (\ln_+ N_i)^{n-2}(1+o(1))
\\ & =
c_{n-1,\ell} \operatorname{flag}(P) (\ln N)^{n-2} (1+o(1)). 
\end{align*}
For the case $\ell=n-1$ observe that each $(n-2)$-dimensional face is contained in precisely two $(n-1)$-dimensional facets. Thus each $(n-2)$-dimensional face of $P_N$ in a facet of $P$ gives rise to a second facet of $P_N$ which is the convex hull with one random point in another facet of $P$. This shows 
\begin{align*}
\E f_{n-1}(P_N)  
& \geq
\sum_{F_i} \E f_{n-2}(P_N \cap F_i)  
\\ & =
c_{n-1,n-2} \operatorname{flag}(P) (\ln N)^{n-2} (1+o(1))
\end{align*}
for general polytopes $P$. 

This sheds some light on the geometry of $P_N$. The number of those facets of the random polytope 
that are not contained in the boundary of $P$ are  already of the same order as all facets
that have one vertex in one facet of $P$ and all  the others 
in another one. In fact it follows from our proof that for simple polytopes the main contribution comes from those facets of $P_N$ whose vertices are on \emph{precisely two facets} of $P$.
We refer to  the end of Section \ref{ssec:mainterm} for the details.

Surprisingly this is no longer true for the expectation of the volume difference. Here the main contribution comes from \emph{all facets} of $P_N$.
And -- to our big surprise -- the volume difference contains no logarithmic factor. This is in sharp contrast to the results for random points inside convex sets and shows that the phenomenon mentioned above does not hold for more general random polytopes.
\begin{theorem}\label{th:Vol}
For the expected volume difference between a simple polytope $P \subset \R^n$ and the random polytope $P_N$ with vertices chosen from the boundary of $P$, we have
$$
\E (V_n(P)-V_n(P_N)) =
c_{n,P} N^{- \frac n{n-1}} (1+ O(N^{- \frac 1{(n-1)(n-2)}}))
$$ 
with some $c_{n,P}>0$ depending on $n$ and $P$.
\end{theorem}

Intuitively, the difference volume for a random polytope whose vertices
are chosen from the boundary should be smaller than the one whose
vertices are chosen from the body. Our result confirms this for $N$ sufficiently large. The first one is of the order $N^{-\frac{n}{n-1}}$ compared to $N^{-1} (\ln N)^{n-1}$.
It is well known that for uniform random polytopes in the interior of a convex set the expected missed volume is minimized for the ball for $N$ large \cite{Bl2, Gr1, Gr2}, a smooth convex set, and -- in the planar case -- maximized by a triangle \cite{Bl2, DL, Gi} or more generally by polytopes \cite{BL}. Hence one should also compare the result of Theorem \ref{th:Vol} to the result of choosing random points on the boundary of a smooth convex set. This  clearly leads to a random polytope with $N$ vertices. And by results of Sch\"utt and Werner \cite{SW5}, see also Reitzner \cite{Re4}, the expected volume difference is of order 
$N^{- \frac 2{n-1}}$ which is smaller as the order in \eqref{eq:EVi} as is to be expected, but also surprisingly much bigger than the order $N^{- \frac {n}{n-1}} $ occurring in Theorem \ref{th:Vol}.

We give a simple argument that shows that the volume difference
between the cube and a random polytope is at least of the order
$N^{-\frac{n}{n-1}}$. We denote by $e_1, \dots, e_n$ the unit vectors of the standard orthonormal basis in $\mathbb{R}^n$. We consider the cube $C^{n}=[0,1]^{n}$
and the subset of the boundary
\begin{equation}\label{eq:simpl-cube}
\partial C^{n}\cap H_{+}\left(\left(\frac{(n-1)!}{n N}\right)^{\frac{1}{n-1}}, (1,\dots,1)\right)
=\bigcup_{i=1}^{n}\left(\frac{(n-1)!}{n N}\right)^{\frac{1}{n-1}}
[0,e_{1},\dots,e_{i-1},e_{i+1},\dots,e_{n}],
\end{equation}
where $H_+(h,u)=\{ x\colon \langle x,u \rangle \geq h \}$.
These sets  are the union of small simplices in the facets of the cube close to the vertices. Then 
$$
\frac{1}{N}=\l_{n-1} \left(\bigcup_{i=1}^{n}\left(\frac{(n-1)!}{n N}\right)^{\frac{1}{n-1}}
[0,e_{1},\dots,e_{i-1},e_{i+1},\dots,e_{n}]\right),
$$
where $\l_{n-1}$ denotes the $(n-1)$-dimensional Hausdorff measure,  and the probability
that none of the points $x_{1},\dots,x_{N}$ is chosen from this set
equals
$$
\left(1-\frac{1}{N}\right)^{N}\sim \frac{1}{e}.
$$
Therefore, with probability approximately $\frac{1}{e}$
the union of the simplices in \eqref{eq:simpl-cube}
is not contained in the random polytope and the difference volume is greater than
$$
\frac{1}{n!}\left(\frac{(n-1)!}{n N}\right)^{\frac{n}{n-1}}
\sim \frac{1}{n}N^{-\frac{n}{n-1}} 
$$
which is in accordance with Theorem \ref{th:Vol}.

\bigskip
The paper is organized in the following way. The next section contains a tool from integral geometry and two asymptotic expansions. The proof of the asymptotic expansions is rather technical and shifted to the end of the paper, 
Appendix \ref{sec:useful-sub}, \ref{sec:main-asymp-exp1} and \ref{sec:main-asymp-exp2}. The third section is devoted to the proofs of Theorem \ref{th:EfN} and Theorem \ref{th:Vol}. There, first we evaluate two formulas for the quantities appearing in Theorems  \ref{th:EfN} and  \ref{th:Vol} and combine them with the necessary asymptotic results. These results are proven in in Sections \ref{ssec:mainterm}--\ref{ssec:error2}, using computations for the moments of the volume of involved random simplices in Section \ref{ssec:rd-simpl}.

Throughout this paper $c_n, c_{\bm m, P, n, \dots}, \ldots$ are generic constants depending on $\bm m$, $P$, $n$, etc. whose precise values may differ from occurrence to occurrence.

\section{Tools }\label{sec:tools} 
We work in the Euclidean space $\R^n$ with inner product $\langle \cdot, \cdot \rangle$ and norm $\| \cdot \|$. We write $H=H(h,u)$ for the hyperplane with unit normal vector $u \in S^{n-1}$ and signed distance $h$ to the origin, 
$H(h,u)= \{x\colon \langle x,u \rangle =h  \} .$ We denote by $H_-=H_- (h,u)=\{ x\colon \langle x,u \rangle \leq h \}$  and by $H_+ = H_+(h,u)=\{ x\colon \langle x,u \rangle \geq h \}$ the
 the  two closed halfspaces bounded by the hyperplane $H$.
For a set $A \subset \R^n$ we write $[A]$ for the convex hull of $A$.

\medskip
In this paper we need a formula for $n$ points distributed on the boundary of a given convex body. A theorem which involves such a formula of Blaschke-Petkantschin type as a special case has been proved by Z\"ahle \cite{Zae}, Theorem 1. We  state it  here only for a $(d-1)$-dimensional measurable set $X$,  which is what we need in the following. 
Measurability is with respect to the $(d-1)$-dimensional Hausdorff measure.
For two hyperplanes $H_1, H_2$ let $ J(H_1, H_2)$ be the length of the projection of a unit interval in $H_1 \cap (H_1 \cap H_2)^\perp $ onto $ H_2^\perp $, or $J(H_1, H_2)=0$ if $H_1 \parallel H_2$.
Observe that $J(H(h_1, u_1), H(h_2, u_2))$ is just the length of the projection of $u_2$ onto 
$H_1$, which equals $\sin \sphericalangle (u_1, u_2)$.

\medskip 
\begin{theorem}[Z\"ahle \cite{Zae}]
Let $X \subset \R^n$ be a $(d-1)$-dimensional measurable set, and let $g: (\R^n)^{n-1} \to [0, \infty) $ be a measurable function. 
Then there is a constant $\b$ such that 
\begin{align*}
\int \limits_{S^{n-1}\times \R} \ 
&
\int \limits_{X \cap H} \!\cdots\! \int \limits_{ X \cap H} 
\I(x_1, \dots, x_{n} \mbox{ in general position}) g(x_1,\dots, x_n) 
 \, d x_1  \cdots d x_n  \, dh \, du
\\[4ex]  &=
\frac{\b}{(n-1)!}  
\int \limits_{X} 
\cdots \int \limits_{X}
\I (x_1, \dots, x_n \ \mbox{ in general position})  g(x_1,\dots, x_n) 
\\[0ex]  & \hskip3.5cm
\l_{n-1} ( [x_1, \dots , x_n])^{-1}
\prod\limits_{j=1}^n J (T_{x_j},H(x_1, \dots, x_n))
\, d x_1 \cdots d x_n   
\end{align*}
with
$dx, \, du,\ dh$ denoting integration with respect to the Hausdorff measure on the respective range of integration, and where $T_{x}$ denots the (generalized) tangent hyperplane at $x$ to $X$, and the hyperplane $H(x_1, \dots, x_n)$ is the affine hull of $x_1, \dots, x_n$. 
\end{theorem}

In our case $X$ is the boundary of a polytope $P$ and thus $J (T_{x_j},H(x_1, \dots, x_n))=0$ if all points are on the same facet of $P$.
To exclude this from the range of integration, we write $(\bd P)^n_{\neq}$ for the set of all $n$-tuples $ x_1, \dots, x_n \in \bd P$ which are not all contained in the same facet. Also, ignoring sets of measure zero,  we may assume that $ x_1, \dots, x_n $ are in general position when integrating with respect to Hausdorff measure on $(\bd P)^n_{\neq}$. And, again ignoring sets of measure zero, a hyperplane $H(h,u)$ meets $\bd P$ at least in $d$ facets, or $\bd P \cap H(h,u)=\emptyset$.

\medskip 
\begin{lemma}
Let $g(x_1, \dots, x_n)$ be a continuous function. 
Then there is a constant $\b$ such that 
\begin{align}\label{eq:Zaehle}
\int \limits_{(\bd P)^n_{\neq} } 
& 
g(x_1,\dots, x_n) \, d x_1 \cdots d x_n  
\\  =& \nonumber
\b^{-1} (n-1)!  
\int \limits_{S^{n-1}} \! \int \limits_{\R}  \ 
\int \limits_{(\bd P \cap H)^n_{\neq}} 
g(x_1,\dots, x_n) 
\l_{n-1} ( [x_1, \dots , x_n])  \prod\limits_{j=1}^n J (T_{x_j},H)^{-1}  \, d x_1  \cdots d x_n  \, dh \, du
\end{align}
with
$dx, \, du,\ dh$ denoting integration with respect to the Hausdorff measure on the respective range of integration.
\end{lemma}

\bigskip
One of the essential ingredients of our proof are two asymptotic expansions of the function 
\begin{eqnarray}\label{GrossJ}
\cJ(\bm l)
&=&
\int \limits_0^1 \dots  \int\limits_0^{1} 
\left(1- \a \sum_{i=1}^{n} \prod_{j \neq i} t_j \right)^{N-n} 
\prod_{i=1}^n t_i^{n-2-l_i}  dt_1 \dots dt_n   
\end{eqnarray}
of $\bm l=(l_{1},\dots,l_{n})$
as $N \to \infty$. Here,  $l_i \in \mathbb{R}$, $l_i >0$ for all $i$ and $\alpha \in \mathbb{R}$, $\alpha >0$.
We need it for the computation of the expectations of the number of facets and of the expected volume difference. The proof of these results is rather technical and lengthy, and will be found in Section~\ref{sec:useful-sub},  Section~\ref{sec:main-asymp-exp1} and in Section~\ref{sec:main-asymp-exp2}.

\begin{lemma}\label{le:main-asymp-exp1}
Assume that $n \geq 2$, $0 < \a < \frac{1}{n}$ and that $\bm l=(l_1, \dots, l_n)$, $L=\sum_{i=1}^n l_i $, with $ n-1 > l_{i} > \frac {L}{n-1} -1 $  for all $i=1, \dots, n$. 
Then 
$$
\cJ(\bm l)
=
\a^{-n + \frac{L}{n-1} }  (n-1)^{-1} 
\left(\prod_{i=1}^n \Gamma \left( l_{i} -  \frac {L}{n-1} +1 \right) \right)
\ N ^{-n+\frac{L}{n-1} } 
\left(1+ O\left(N^{-\frac 1{n-2}(\min_k l_k- \frac L{n-1} +1)} \right) \right)
$$
as $N \to \infty$, where the implicit constant in $O(\cdot)$ may depend on $\a$.
\end{lemma}

\begin{lemma}\label{le:main-asymp-exp2}
Assume that $n \geq 2$, $0 < \a \leq \frac 1 {2n}$ and $\bm l=(l_1, \dots, l_n)$, $L=\sum_{i=1}^n l_i $, 
with $ n-1 > l_{i} \geq \frac {L}{n-1} -1 $ for all $i=1, \dots, n$. If for at least three different indices $i, j,k$ 
we have the strict inequality that $ l_{i}, l_j, l_k >  \frac {L}{n-1} -1 $, 
then 
$$
\cJ(\bm l)
=
O\left(  N^{-n+ \frac L{n-1}} (\ln N)^{n-3} \right)
$$
as $N \to \infty$, where the implicit constant in $O(\cdot)$ may depend on $\a$.
If for exactly two different indices $i, j$ we have the strict inequality that $ l_{i}, l_j >  \frac {L}{n-1} -1 $ and equality $l_k = \frac {L}{n-1} -1 $ for all other $l_k$, 
then 
$$
\cJ(\bm l)
=
c_n \a^{-n+\frac L{n-1}} \Gamma(l_i- \frac L{n-1} +1)\Gamma(l_j- \frac L{n-1} +1)  N^{-n+ \frac L{n-1}} (\ln N)^{n-2} (1+O((\ln N)^{-1})) 
$$
as $N \to \infty$ with $c_n>0$, where the implicit constant in $O(\cdot)$ may depend on $\a$.
\end{lemma}

\section{Proof of Theorem \ref{th:EfN} and Theorem \ref{th:Vol} }\label{sec:proof}


\subsection{The number of facets}\label{ssec:idea-facet}
Let $P \subset \R^n$ be a simple polytope, and assume w.l.o.g. that the surface area satisfies $\l_{n-1} ( \bd P)=1$. As usual denote by $\cF_k(P)$ the set of $k$-dimensional faces of $P$. Choose random points $X_1, \dots , X_N$ on the boundary of $P$ with respect to Lebesgue measure, and denote by $P_N=[X_1, \dots, X_N]$  their convex hull. In general $\mathcal{F}_{n-1}(P_N)$ consists of facets contained in facets of $P$ and facets which are formed by random points on different facets of $P$. The latter  facets are simplices, almost surely.
The number of facets contained in $\partial P$ is bounded by the number of facets of $P$ and thus by a constant. 
Hence we assume in the following that $(X_1, \dots, X_n) \in (\bd P)_{\neq}^n$.
The convex hull of such points $X_{i}$, $i \in I=\{i_1, \dots, i_n\}$ forms a facet $[X_{i_1}, \dots, X_{i_n}]$ of $P_N$ if their affine hull does not intersect the convex hull of the remaining points $[\{X_{j}\}_{j \notin I}]$. 
\begin{eqnarray*}
\E f_{n-1}(P_N)
&=& 
\E\, \sum_{I \subset \{ 1, \dots, N\}, |I|=n} \I  \left( \aff [\{ X_i \}_{i \in I}]  \cap [\{X_j \}_{j \notin I} ] = \emptyset,\, \{ X_i \}_{i \in I} \in (\bd P)^n_{\neq} \right) +O(1)
\\ &=& 
\binom N n \  \E\,  \I  \left( \aff [X_1, \dots, X_n]  \cap [X_{n+1}, \dots, X_N] = \emptyset,\, \{ X_i \}_{i \leq n} \in (\bd P)^n_{\neq} \right) +O(1)
\end{eqnarray*}
To simplify notation we set $H= \aff [X_1, \dots, X_n]$.
If the points $X_1, \dots, X_n$ form a facet then their affine hull is a supporting hyperplane $H=H(h,u)$ of the random polytope $P_N$. 
The unit vector $u$ is the unit outer normal vector of this facet. Then the halfspace $H_-=H_-(h,u)=\{ x\colon \langle x,u \rangle \leq h \}$ bounded by the hyperplane $H$ contains the random polytope $P_N$. The probability that $X_{n+1}, \dots, X_N$  are contained in $H_- $ equals 
$$ \l_{n-1}(\bd P \cap H_- )^{N-n} = (1-\l_{n-1}( \bd P \cap H_+ ))^{N-n} ,$$
thus 
\begin{eqnarray*}
\E f_{n-1}(P_N)
&=& 
\binom N n  \E\Big( (1- \l_{n-1}( \bd P \cap H_+ ))^{N-n} \I(\{ X_i \}_{i \leq n} \in (\bd P)^n_{\neq}) \Big) + O(1).
\end{eqnarray*}
Denote by $H(P,u)$ a support hyperplane with normal $u$,  supporting $P$ in $v \in \cF_0(P)$.  Then the normal cone 
$\cN(v,P)$ is defined as (see e.g., \cite{SchneiderBook}), 
$$
\cN(v,P) = \{u \in \mathbb{R}^n \setminus \{0\}: v \in H(P,u) \cup \{0\} \}.
$$
With probability one the vector $u$ is contained in the interior of one of the normal cones $\cN(v,P)$ 
of the 
vertices $v \in \cF_0(P)$ of $P$ because the boundaries of the normal cones have measure $0$. Hence
\begin{align*}
\E f_{n-1}(P_N)
=& 
\sum_{v \in \cF_0(P)} \binom N n  \E\Big( (1- \l_{n-1}( \bd P \cap H_+ ))^{N-n} \I  ( u \in \cN(v, P),\, \{ X_i \}_{i \leq n} \in (\bd P)^n_{\neq}) \Big)
+O(1)
\\ =& 
\sum_{v \in \cF_0(P)}   \binom N n  \idotsint\limits_{(\bd P)_{\neq}^n} 
(1- \l_{n-1}(  \bd P \cap H_+ ))^{N-n} \I  ( u \in \cN(v, P))  \, d x_1 \dots d x_n
+O(1) .
\end{align*}
Now we fix a vertex $v$. Since $P$ is simple, $v$ is contained in precisely $n$ facets $F_1, \dots, F_n$. 
There is an affine transformation $A_v$ which maps $v$ to the origin and the $n$ edges $[v, v_i]$ containing $v$ onto  segments $[0, s_i e_i]$ on the coordinate axis. Here we have a free choice for the $n$ parameters $s_i>0$ which we will fix soon. We assume $s_i \geq n$, $i=1, \dots, n$, which implies
$$[0, 1]^n \subset A_v P. $$
The image measure $\l_{n-1, A_v}$ of the Lebesgue measure $\l_{n-1 }$ on the facets of $P$ under the affine transformation $A_v$ is - up to a constant  - again Lebesgue measure, where the constant may differ for different facets. We choose the $n$ parameters $s_i \geq n$ in such a way that the constant equals the same $a_v>0$ for the $n$ facets $F_1, \dots, F_n$ containing $v$,
$$
\l_{n-1} (F_i) = a_v \l_{n-1} (A_v F_i)
.$$
Note that the last condition means that for all such facets $F_i$ and all measurable 
$B \subset A_v F_i$, 
\begin{equation}\label{eq:imageaffineI}
\l_{n-1, A_v} ( B) =  \l_{n-1} (A_v^{-1} B) = a_v \l_{n-1}(B) .
\end{equation}
Note also that $[0, 1]^{n-1}  \subset A_v F_i$, $i=1, \dots,n$ and thus by (\ref{eq:imageaffineI}),
\begin{equation}\label{alpha-klein}
n=\sum_{i=1}^n \l_{n-1} \left( [0, 1]^{n-1}\right) \leq \frac{1}{a_v} \sum_{i=1}^n \l_{n-1} \left(F_i\right) \leq  \frac{1}{a_v} S(P) = \frac{1}{a_v} .
\end{equation}
Such a uniform bound on $a_v$ is needed in Subsection \ref{ssec:mainterm} so that $\alpha = \frac{a_v}{(n-1)!} \leq \frac{1}{2n}$.

To keep the notation short, we write 
$d_{A_v} x = d \l_{n-1,A_v} (x)$ 
for integration with respect to this image measure of the Lebesgue measure on 
$\bd P$ under the map $A_v$. Equation (\ref{eq:imageaffineI}) shows that for $x \in A_v F_i$,
\begin{equation}\label{eq:imageaffineII}
 d_{A_v} x = a_v dx 
\end{equation}
for $i=1, \dots , n$, where $d x$ is again shorthand for Lebesgue measure (or equivalently Hausdorff measure) on the respective facet $F_i$.
This yields
\begin{eqnarray*}
\E f_{n-1}(P_N)
&=& 
\sum_{v \in \cF_0( P)}   \binom N n  \idotsint\limits_{(\bd A_v  P)_{\neq}^n}  (1- \l_{n-1, A_v} (\bd A_v P \cap H_+ ))^{N-n} 
\\ && \hskip4cm \times 
\I  ( u \in \cN(0, A_v P)) \, d_{A_v} x_1 \dots d_{A_v} x_n
\ +O(1) .
\end{eqnarray*}
We use Z\"ahle's formula \eqref{eq:Zaehle} which transforms the integral over the points $x_i \in \bd P$ into an integral over all the hyperplanes $H=H(h,u)$, $u \in S^{n-1}$, $h \in \R$, and integrals over $\bd P \cap H$:
\begin{eqnarray*}
\E f_{n-1}(P_N)
&=& 
\sum_{v \in \cF_0( P)}   \binom N n  
\b^{-1} (n-1)!  
\int \limits_{S^{n-1}} \! \int \limits_{\R}  \ 
\idotsint\limits_{(\bd A_v  P \cap H)_{\neq}^n}  (1- \l_{n-1, A_v} (\bd A_v P \cap H_+ ))^{N-n} 
\\ &&
\l_{n-1} ( [x_1, \dots , x_n])  \prod\limits_{j=1}^n J (T_{x_j},H)^{-1}  
\I  ( u \in \cN(0, A_v P))  \, d_{A_v} x_1 \dots d_{A_v} x_n \, dh du
\\ && 
+O(1).
\end{eqnarray*} 
We have $\cN(0, A_v P)=-S_{+}^{n-1}$ where $S_+^{n-1}= S^{n-1} \cap \R_+^n$.
The condition $\I  ( u \in \cN(0, A_v P))$ will be taken into account in the range of integration in the form $ u \in - S_+^{n-1}$. 
Now we fix $u$ and split the integral into two parts. In the first one $H_- $ contains all the unit vectors $e_i$. We write this condition in the form 
$$\max_{i=1, \dots, n} u_i \leq h \leq 0 . $$
Note that $h \leq 0$, since $u \in  - S_+^{n-1}$. 
The second part is over 
$ h \leq \max_{i=1, \dots, n} u_i . $
Thus the expected number of facets is
\begin{align*}
\E f_{n-1}(P_N)  
&= 
\sum_{v \in \cF_0( P)}   \binom N n  
\b^{-1} (n-1)!  
\Bigg( 
\int \limits_{-S_+^{n-1}}   \int \limits_{\max u_i}^0  \ 
(1- \l_{n-1, A_v} ( \bd \R_+^n \cap H_+))^{N-n}  \times  
\\ & \hskip2.5cm  
\idotsint\limits_{(\bd \R_+^n \cap H)_{\neq}^n} 
\l_{n-1} ( [x_1, \dots , x_n])  \prod\limits_{j=1}^n J (T_{x_j},H)^{-1}  
\, d_{A_v} x_1 \dots d_{A_v} x_n \, dh du
\\[1ex] & \hskip4cm +
\int \limits_{-S_+^{n-1}}   \int \limits_{-\infty}^{\max u_i}
 (1- \l_{n-1, A_v} (\bd A_v P \cap H_+ ))^{N-n}   \times
\\ & \hskip2.5cm 
\idotsint\limits_{(\bd A_v  P \cap H)_{\neq}^n} 
 \l_{n-1} ( [x_1, \dots , x_n])  \prod\limits_{j=1}^n J (T_{x_j},H)^{-1}  
 \, d_{A_v} x_1 \dots d_{A_v} x_n \, dh du
\Bigg)
\\ & 
+O(1), 
\end{align*}
where we have used that in the first case $\bd  A_v P \cap H_+ =  \bd \R_+^n \cap H_+$.
The substitution $u \mapsto -u$ and $h \mapsto -h$ yields the more convenient formula 
\begin{eqnarray*}
\E f_{n-1}(P_N)  
&=& 
\sum_{v \in \cF_0( P)}   \binom N n  
\b^{-1} (n-1)!  
( I_v^1+E_v^1) \ + O(1)
\end{eqnarray*}
with 
\begin{eqnarray*}
I_v^1
&=& 
\int \limits_{S_+^{n-1}}   \int \limits_0^{\min u_i}  \ 
(1- \l_{n-1, A_v} (\bd  \R_+^n \cap H_-))^{N-n}  
\\ && \hskip3.5cm  
\idotsint\limits_{(\bd \R_+^n \cap H)_{\neq}^n} 
\l_{n-1} ( [x_1, \dots , x_n])  \prod\limits_{j=1}^n J (T_{x_j},H)^{-1}  
\, d_{A_v} x_1 \dots d_{A_v} x_n \, dh du  ,
\\[1ex] 
E_v^1
&=&
\int \limits_{S_+^{n-1}}   \int \limits_{\min u_i}^{\infty}
 (1- \l_{n-1, A_v} (\bd  A_v P \cap H_-))^{N-n}  
\\ && \hskip3.5cm 
\idotsint\limits_{(\bd A_v  P \cap H)_{\neq}^n} 
 \l_{n-1} ( [x_1, \dots , x_n])  \prod\limits_{j=1}^n J (T_{x_j},H)^{-1}  
 \, d_{A_v} x_1 \dots d_{A_v} x_n \, dh du   .
\end{eqnarray*}
The asymptotically dominating term will be $I_v^1$. Using \eqref{eq:imageaffineI} and \eqref{eq:imageaffineII} for $I_v^1$ we have 
\begin{eqnarray} \label{eq:defI0}
I_v^1
&=& 
a_v^n \int \limits_{S_+^{n-1}}   \int \limits_0^{\min u_i}  \ 
(1- a_v \l_{n-1}(\bd \R_+^n \cap H_-))^{N-n}  
\\  && \hskip3.5cm \nonumber  
\times \idotsint\limits_{(\bd \R_+^n \cap H)_{\neq}^n} 
\l_{n-1} ( [x_1, \dots , x_n])  \prod\limits_{j=1}^n J (T_{x_j},H)^{-1}  
\, d x_1 \dots d x_n \, dh du  
.
\end{eqnarray}
In Section \ref{ssec:mainterm} we will determine the precise asymptotics. Equation \eqref{eq:I_v^1} will tell us that
$$ I_v^1 = c_n N^{-n} (\ln N)^{n-2} (1+O((\ln N)^{-1})) $$
with some constant $c_n >0$ as $N \to \infty$. 

The error term $E_v^1$ can be estimated by using the fact that there are constants $\overline a, \underline a$ such that 
\begin{equation}\label{aquer}
 \underline a \l_{n-1} (B) \leq \l_{n-1, A_v} ( B) \leq \overline a \l_{n-1}(B)
\end{equation}
for all $v \in \cF_0(P)$ and all $B \subset \bd P$. 
This shows 
\begin{eqnarray} \label{eq:defE0}
 E_v^1
&\leq &
(2\overline a)^n \int \limits_{S_+^{n-1}}   \int \limits_{\min u_i}^{\infty}
 (1- \underline a \l_{n-1}(\bd  A_v P \cap H_-))^{N}  
\\ && \hskip3.5cm \nonumber
\idotsint\limits_{(\bd A_v  P \cap H)_{\neq}^n} 
 \l_{n-1} ( [x_1, \dots , x_n])  \prod\limits_{j=1}^n J (T_{x_j},H)^{-1}  
 \, d x_1 \dots d x_n \, dh du   .
\end{eqnarray}
In Section \ref{ssec:error1} we will show that this is of order $O(N^{-n} (\ln N)^{n-3})$, see \eqref{eq:E_v^1}.
This implies 
\begin{eqnarray}\label{eq:result-facets}
\E f_{n-1}(P_N)  
&=& \nonumber
\sum_{v \in \cF_0( P)}   \binom N n  
\b^{-1} (n-1)!  c_n N^{-n} (\ln N)^{n-2} (1+O((\ln N)^{-1}))
\\ &=& 
c_n f_0( P) (\ln N)^{n-2} (1+O((\ln N)^{-1}))
\end{eqnarray}
with some $c_n >0$ which is Theorem \ref{th:EfN}.

\subsection{The volume difference}\label{ssec:idea-vol}
We are interested in the expected volume difference 
$$ 
\E (V_n(P)-V_n(P_N)).
$$
With probability one the random polytope $P_{N}$ has the following property: 
For each facet $F \in \cF_{n-1}(P_N)$ that is not contained in a facet of $P$ there exists a unique vertex $v \in \cF_0(P)$, such that the outer unit normal vector $u_{F}$ of $F$ is contained in the normal cone $\cN( v, P)$, or equivalently the hyperplane $H$ containing $F$ is parallel to a supporting hyperplane to $P$ at $v$. Clearly all the sets $[F, v]$ are contained in $P \setminus P_N$ and they have pairwise disjoint interiors. This is immediate in dimension two, and holds in arbitrary dimensions because it holds for all two-dimensional sections through $P$ and $P_N$ containing two vertices of $P$.
We set 
\begin{equation}\label{def:CN-DN}
C_N= \bigcup_{v \in \cF_{0}(P)}  \bigcup_{ \substack{F \in \cF_{n-1}(P_N) \\ u_F \in \cN(v,P) \\ F \nsubseteq \bd P  } }  [F, v],\hskip 20mm  
D_N= P\setminus( P_N \cup C_N ) 
\end{equation}
where $D_N$  is the subset of $P \setminus P_N$ not covered by one of the simplices $[F, v]$ with $u_F \in \cN(v,P)$. We have 
\begin{eqnarray*}\label{eq:EdiffV-basic}
\E (V_n(P)-V_n(P_N))
&=& 
\E V_n(C_N)+ \E V_n(D_N) 
\\ &=& 
\E \left(
\sum\limits_{v \in \cF_0(P)} \sum\limits_{F \in \cF_{n-1}(P_N)}
V_n([F,v]) \I(u_F \in \cN(v,P))  \right) 
\ + \E V_n(D_N) .
\end{eqnarray*}
For the first summand we follow the approach already worked out in detail in the last section.
The convex hull $[X_{i_1}, \dots, X_{i_n}]$ forms a facet of $P_N$ if their affine hull does not intersect the convex hull of the remaining point, and to  simplify notation we set $u=u_F$ and $H(h,u)=H= \aff [X_1, \dots, X_n]$. The halfspace $H_-$ contains the random polytope $P_N$, and the probability that $X_{n+1}, \dots, X_N$  are contained in $H_- $ equals 
$$ \l_{n-1}(\bd P \cap H_- ))^{N-n} = (1-\l_{n-1}( \bd P \cap H_+ ))^{N-n}.$$
Thus 
\begin{align*}
\E V_n & (C_N)
\\=& 
\sum_{v \in \cF_0(P)} \binom N n  \E\Big( (1- \l_{n-1}( \bd P \cap H_+ ))^{N-n} \I  ( u \in \cN(v, P),\, \{ X_i \}_{i \leq n} \in (\bd P)^n_{\neq}) V_n[X_1, \dots, X_n, v] \Big) 
\\ =& 
\sum_{v \in \cF_0(P)}   \binom N n  \idotsint\limits_{(\bd P)_{\neq}^n} 
(1- \l_{n-1}(  \bd P \cap H_+ ))^{N-n} \I  ( u \in \cN(v, P))  V_n[x_1, \dots, x_n, v] \, d x_1 \dots d x_n
.
\end{align*}
We fix $v$ and use the affine transformation $A_v$ defined in the last section which maps $v$ to the origin and the edges onto the coordinate axes.
The transformation rule yields
\begin{eqnarray*}
\E V_n(C_N)
&=& 
\sum_{v \in \cF_0( P)}   \binom N n  \idotsint\limits_{(\bd A_v  P)_{\neq}^n}  (1- \l_{n-1, A_v} (\bd A_v P \cap H_+ ))^{N-n} \I  ( u \in \cN(0, A_v P)) 
\\ && \hskip3cm 
V_n [A_v^{-1} x_1, \dots, A_v^{-1} x_n,0]
\, d_{A_v} x_1 \dots d_{A_v} x_n .
\end{eqnarray*}
The volume of the simplex $[\{A_v^{-1} x_i\}_{i=1, \dots, n}, 0]$ is a constant $d_v=\det A_v^{-1}$ times the volume of $[\{x_i\}_{i=1, \dots, n}, 0]$ which equals $n^{-1}$ times the height $|h|$ times the volume of the base $[\{ x_i\}_{i=1, \dots, n}]$ .
By Z\"ahle's formula \eqref{eq:Zaehle} we obtain
\begin{eqnarray*}
\E V_n(C_N)
&=& 
\sum_{v \in \cF_0( P)}  d_v  \binom N n  
\b^{-1} \frac{(n-1)!}n  
\int \limits_{S^{n-1}} \! \int \limits_{\R}  \ 
\idotsint\limits_{(\bd A_v  P \cap H)_{\neq}^n}  (1- \l_{n-1, A_v} (\bd A_v P \cap H_+ ))^{N-n} 
\\ &&
|h| \l_{n-1} ( [x_1, \dots , x_n])^2  \prod\limits_{j=1}^n J (T_{x_j},H)^{-1}  
\I  ( u \in \cN(0, A_v P))  \, d_{A_v} x_1 \dots d_{A_v} x_n \, dh du .
\end{eqnarray*} 
We  split the integral into the two parts $\max_{i=1, \dots, n} u_i \leq h \leq 0 $ and $ h \leq \max_{i=1, \dots, n} u_i $ and substitute $u \mapsto -u,\ h \mapsto -h$.
The main part of the expected volume difference is
\begin{eqnarray*}
\E V_n(C_N)  
&=& 
\sum_{v \in \cF_0( P)} d_v  \binom N n  
\b^{-1} \frac{(n-1)!}n  
( I_v^2+E_v^2)
\end{eqnarray*}
with 
\begin{eqnarray}\label{eq:defI1}
I_v^2
&=& 
a_v^n \int \limits_{S_+^{n-1}}   \int \limits_0^{\min u_i}  \ 
(1- a_v \l_{n-1}(\bd \R_+^n \cap H_-))^{N-n} 
\\ && \hskip3.5cm \nonumber  
h \idotsint\limits_{(\bd \R_+^n \cap H)_{\neq}^n} 
 \l_{n-1} ( [x_1, \dots , x_n])^2  \prod\limits_{j=1}^n J (T_{x_j},H)^{-1}  
\, d x_1 \dots d x_n \, dh du  
\\[1ex] 
E_v^2
&=&
\int \limits_{S_+^{n-1}}   \int \limits_{\min u_i}^{\infty}
 (1- \l_{n-1, A_v} (\bd  A_v P \cap H_-))^{N-n}  
\\ && \hskip3.5cm \nonumber
h \idotsint\limits_{(\bd A_v  P \cap H)_{\neq}^n} 
 \l_{n-1} ( [x_1, \dots , x_n])^2  \prod\limits_{j=1}^n J (T_{x_j},H)^{-1}  
 \, d_{A_v} x_1 \dots d_{A_v} x_n \, dh du   .
\end{eqnarray}
The asymptotically dominating term will be $I_v^2$. In Section \ref{ssec:mainterm} we determine the precise asymptotics. Equation \eqref{eq:I_v^2} will tell us that 
$$ I_v^2 = c_n  a_v^{- \frac n{n-1}} N^{-n- \frac n{n-1}} (1+ O(N^{- \frac 1{(n-1)(n-2)}})) $$
with some constant $c_n >0$ as $N \to \infty$. 

The error term $E_v^2$ can be estimated by
\begin{eqnarray} \label{E-Term}
 E_v^2
&\leq &
(2\overline a)^n \int \limits_{S_+^{n-1}}   \int \limits_{\min u_i}^{\infty}
 (1- \underline a \l_{n-1}(\bd  A_v P \cap H_-))^{N}  
\\ && \hskip3.5cm \nonumber
h \idotsint\limits_{(\bd A_v  P \cap H)_{\neq}^n} 
 \l_{n-1} ( [x_1, \dots , x_n])^2  \prod\limits_{j=1}^n J (T_{x_j},H)^{-1}  
 \, d x_1 \dots d x_n \, dh du,
\end{eqnarray}
where $\overline a,  \underline a$ are as in (\ref{aquer}).
In Section \ref{ssec:error1}, equation \eqref{eq:E_v^2} we show that 
\begin{eqnarray*}
E_v^2  
=
O( N^{-n-\frac {n-1}{n-2} }).
\end{eqnarray*}

\bigskip
It remains to estimate 
$$
\E V_n(D_N)= \E \left( P\setminus\left( P_N \cup \bigcup_{F \in \cF_{n-1}(P_N)}  [F, v_F] \right) \right) . 
$$

The following argument is proved in detail in the paper of Affentranger and Wieacker \cite[p.302]{AW} and will only be sketched here. 

If $y \in D_N$, then the normal cone $\cN(y, [y,P_N])$ is not contained in any of the normal cones $\cN(v,P)$ of $P$, $v \in \cF_0(P)$. Hence $\cN(y, [y,P_N])$ meets at least two neighbouring normal cones $\cN(v_1,P), \cN(v_2,P)$, and thus the normal cone of the edge $e=[v_1,v_2] \in \cF_1(P)$. This implies that there exists a supporting hyperplane $H$ of $P$ with $H \cap P=e$ with the property that the parallel hyperplane through $y$ does not meet $P_N$.

We apply an affine map $A_e$ similar to the one defined above which maps $e=[v,w]$ to the unit interval $[0, e_n]$, $v$ to the origin, and the image of other edges containing $v$ contain the remaining unit intervals $[0,e_i]$. After applying this map the situation described above is the following: for $x=(x_1, \dots, x_n)= A_e y \in A_e D_N$ the supporting hyperplane $ A_e H =H(0, u)$ to $A_e P$ intersects $A_e P$ in the edge $[0, e_n]$.
The parallel hyperplane $H(\langle x,u\rangle ,u)$ contains $x$ and cuts off from $A_e P$ a cap disjoint from $A_e P_N$. This cap contains the simplex
$[0, \min(1,x_1) e_1, \dots, \min(1, x_{n-1}) e_{n-1}, e_n  ]$. 
Hence if $x \in A_e D_N$ then 
$$ [0, \min(1,x_1) e_1, \dots, \min(1, x_{n-1}) e_{n-1}, e_n  ] \cap A_e P_N= \emptyset  . $$
The probability of this event is given by 
\begin{eqnarray*}
\P([0, x_1 e_1, \lefteqn{
\dots, x_{n-1} e_{n-1}, e_n  ] \cap A_e P_N= \emptyset ) 
= 
}&&
\\ &=&
\left( 1- \l_{n-1} (A_e^{-1} ( \bd \mathbb{R}_+^n \cap [0, \min(1,x_1) e_1, \dots, \min(1,x_{n-1}) e_{n-1}, e_n  ]) )  \right)^{N}
\\ &\leq &
\left( 1- \underline a\l_{n-1} (\bd \mathbb{R}_+^n \cap [0, \min(1,x_1) e_1, \dots, \min(1,x_{n-1}) e_{n-1}, e_n  ])   \right)^{N}
\end{eqnarray*}
with some $\underline a >0$. We denote by $d_e$ the involved Jacobian of $A_e^{-1}$ and by $\overline d$ the maximum of $d_e$.
This implies the estimate
\begin{eqnarray*}
\E V_n(D_N) 
& = &
\int\limits_{P} \P(x \in D_N) dx
\\ & \leq  &
\sum_{e \in \cF_1(P)} d_e \int\limits_{A_e P} \left( 1- \underline a\l_{n-1} ( \bd \mathbb{R}_+^n \cap [0, \min(1,x_1) e_1, \dots, \min(1,x_{n-1}) e_{n-1}, e_n  ])   \right)^{N} dx 
\\ & \leq  &
f_1(P) \, \overline d  \int\limits_{[0,\t]^n} \left( 1- \underline a\l_{n-1} ( \bd \mathbb{R}_+^n \cap [0, \min(1,x_1) e_1, \dots, \min(1,x_{n-1}) e_{n-1}, e_n  ])   \right)^{N} dx 
\end{eqnarray*}
assuming again that $A_e P \subset [0,\t]^n$ for all $e$.
In Section \ref{ssec:error2} we prove that 
\begin{equation}\label{eq:F-estimate}
\E V_n(D_N)
= O(N^{- \frac{n-1}{n-2}})   .
\end{equation}
Combining our results we get
\begin{eqnarray}\label{eq:result-vol-diff}
\E (V_n(P)-V_n(P_N))
&=& \nonumber
\sum_{v \in \cF_0( P)} d_v  \binom N n  
\b^{-1} \frac{(n-1)!}n  
( I_v^2+E_v^2)
+ \E V_n(D_N) 
\\ &=& \nonumber
c_n  \sum_{v \in \cF_0( P)} d_v^2  a_v^{- \frac n{n-1}}   
N^{- \frac n{n-1}} (1+ O(N^{- \frac 1{(n-1)(n-2)}}))
\\ &=& 
c_{n,P} N^{- \frac n{n-1}} (1+ O(N^{- \frac 1{(n-1)(n-2)}}))
\end{eqnarray}
which is Theorem \ref{th:Vol}.

\subsection{Random simplices in simplices}\label{ssec:rd-simpl}

For $u \in S_+^{n-1}$, $h \geq 0$ and $H=H(h,u)$ we set
\begin{equation}\label{eq:def_Ek}
\cE_k(h, u)
=
\idotsint\limits_{(\bd \R_+^n \cap H)_{\neq}^n} 
\l_{n-1} ( [x_1, \dots , x_n])^{k}  \prod\limits_{j=1}^n J (T_{x_j},H(1,u))^{-1}  
\, d x_1 \dots d x_n 
\end{equation}
which is the (not normalized) $k$-th moment of the volume of a random simplex in $ \R_+^n \cap H(h,u)$ 
where the random points are chosen on the boundary of this simplex according to the weight functions 
$J (T_{x_j},H(1,u))^{-1} $. Recall that for almost all $x_j$, $T_{x_j}$ is the supporting hyperplane at $x_j$. In fact it is the coordinate hyperplane which contains $x_j$.

\begin{lemma}\label{le:rdsimpl-in-simpl}
For $k \geq 0$, there are constants $\cE_{k,\bm f}>0$ independent of $u$, such that 
\begin{eqnarray*}
\cE_k(h, u)
&=& 
h^{-(n+k )   }  n^{- \frac k2}(n-1)^{- \frac {n}2}
\sum_{\bm f \in \{ 1, \dots, n\}_{\neq}^n } \ 
\left( \prod_{i=1}^n \frac h {u_i} \right)^{n+k} 
\prod\limits_{i=1}^n \frac{u_{f_i}} h \ \cE_{k, \bm f} \ .
\end{eqnarray*}
\end{lemma}
\begin{proof}
For a point $x_j$ in the coordinate hyperplane $e_i^\perp$,  the weight function $J (T_{x_j},H(1,u))^{-1} $ is the sine 
of the angle between $e_i$ and $u$. Thus
\begin{equation}\label{eq:Jberechnen}
J (T_{x_j},H(1,u))= \| u\vert_{e_i^\perp}\| = \Big( 1- u_i^2 \Big)^{\frac 12}
\end{equation}
and hence is independent of $h$ as long as $u$ is fixed.
In \eqref{eq:def_Ek} we substitute $x_j = h y_j$ with $y_j \in H(1,u)$. The $(n-1)$-dimensional volume is homogeneous of degree $n-1$, hence 
$$ \l_{n-1} ([x_1, \dots, x_n]) = h^{n-1} \l_{n-1} ([y_1, \dots, y_n]) ,$$
and since $x_j$ are in the $(n-2)$-dimensional planes $ \bd \R_+^{n} \cap H(h,u)$ we have
$ dx_j = h^{n-2} dy_j . $
\begin{eqnarray}\label{eq:Ekhu}
\cE_k(h, u)
&=& \nonumber
h^{(n-1)k + n(n-2)}
\idotsint\limits_{(\bd \R_+^n \cap H(1,u))_{\neq}^n} 
\l_{n-1} ([y_1, \dots, y_n])^{k} 
\prod\limits_{j=1}^n J (T_{x_j},H(1,u))^{-1}  \, d y_1  \cdots d y_n  
\\[2ex] &=& 
h^{(n-1)k + n(n-2)}  \cE_k (1,u).
\end{eqnarray}
To evaluate $\cE_k(1,u)$ we condition on the facets in $e_1^\perp, \dots, e_n^\perp$ of $\R_+^n \cap H(1,u)$ from where the random points are chosen. Thus for
$$ \bm  f \in \{1, \dots , n\}^n  $$
we condition on the event 
$ y_i  \in e_{f_i}^\perp$. Because $\{y_1, \dots y_n\} \in (\bd \R_+^n \cap H(1,u))_{\neq}^n$, which means that not all points are contained in the same facet, we may assume that $ \bm f \in \{1, \dots , n\}^n_{\neq}  $ where we remove all $n$-tuples of the form $(i, \dots, i)$ and denote the remaining set by $\{1, \dots , n\}^n_{\neq}  $. Recalling \eqref{eq:Jberechnen}, we obtain
\begin{eqnarray}\label{le:rdsimpl-in-simpl-1}
\cE_k(1, u)
&=& 
\sum_{\bm f \in \{ 1, \dots, n\}_{\neq}^n } \ \prod\limits_{i=1}^n (1-u_{f_i}^2 )^{- \frac 12} 
\\&& 
\times \idotsint\limits_{(\bd \R_+^n \cap H(1,u))_{\neq}^n} 
\l_{n-1} ([y_1, \dots, y_n])^{k} 
\prod\limits_{i=1}^n \I( y_i  \in e_{f_i}^\perp) \, d y_1  \cdots d y_n  .
\nonumber
\end{eqnarray}
A short computation shows that $H(1,u)$ meets the coordinate axis in the points $\frac 1 {u_i} e_i$.
We substitute $z= Ay$, $y = A^{-1} z$, where $A$ is the affine map transforming 
$H(1, \mathbf 1_n)$ into $H(1,u)$. Here  $\mathbf 1_n$ is the vector 
$(1, \dots, 1)^T$. The map is given by 
\begin{equation}\label{eq:defA}
A= \left(
\begin{array}{ccc}
u_1 & 0 \cdots & 0 \\[1ex]
\vdots  & \ddots & \vdots \\[1ex]
0 &  \cdots 0 & u_n 
\end{array}
\right) .
\end{equation}
The volume of the simplex $\R_+^n \cap H_- (1,u)$ is given by 
$\frac 1{n!}  \prod_{i=1}^n \frac 1 {u_i}$, thus the \lq base\rq\ $\R_+^n \cap H(1,u)$ of this simplex 
has $(n-1)$-volume $\frac 1{(n-1)!} \prod_{i=1}^n \frac 1 {u_i}$.
The regular simplex $\R_+^n \cap H (1,\mathbf 1_n)$ has $(n-1)$-volume $\frac 
1{(n-1)!} \sqrt n$.
Hence
$$ \l_{n-1} ([A^{-1} z_1, \dots, A^{-1} z_n])^k = 
n^{- \frac k2}
\left( \prod_{i=1}^n \frac 1 {u_i} \right)^k \l_{n-1} ([ z_1, \dots,  z_n])^k .
$$
The $(n-1)$-volume of the simplex spanned by the origin and the facet of $\bd \R_+^n \cap H(1,u)$ in 
$e_i^\perp$ is given by $\frac 1 {(n-1)!} \prod_{j \neq i} \frac 1{u_j}$, its height by  
$\|(u_1, \dots, u_{i-1}, u_{i+1}, u_n) \|^{-1} = (1- u_i^2)^{-\frac 12}$ and hence the $(n-2)$-volume 
of the facet of $\bd \R_+^n \cap H(1,u)$ in $e_i^\perp$ is 
$$\l_{n-2}(\bd \R_+^n \cap e_i^\perp\cap H(1,u) )= \frac 1{(n-2)!}  (1- u_i^2)^{\frac 12} \prod_{j \neq i} \frac 1{u_j} . $$
Comparing this to the volume $\frac 1{(n-2)!} \sqrt {n-1}$ of the facet of the simplex 
$\bd \R_+^n \cap H(1,\mathbf 1_n)$ in $e_i^\perp$ which equals the volume of $\R_+^{n-1} \cap H (1,\mathbf 1_{n-1})$ shows that the Jacobian in $e_{f_i}^\perp$ of the map $A$ is 
$$
\frac{\l_{n-2}(\bd \R_+^n \cap e_i^\perp\cap H(1,u) )}{\l_{n-2}(\bd \R_+^n \cap e_i^\perp\cap H(1,\mathbf 1_n) )}
=
(n-1)^{- \frac 12} (1- u_{f_i}^2)^{\frac 12} \prod_{j 
\neq f_i} \frac 1{u_j} \I(z_i \in e_{f_i}^\perp)  . 
$$
Combining these Jacobians with (\ref{le:rdsimpl-in-simpl-1}) we obtain
\begin{eqnarray*}
\cE_k(1, u)
&=& 
n^{- \frac {k}2}(n-1)^{- \frac {n}2}
\sum_{\bm f \in \{ 1, \dots, n\}_{\neq}^n } \ 
\left( \prod_{i=1}^n \frac 1 {u_i} \right)^{n+k} 
\prod\limits_{i=1}^n u_{f_i} 
 \\&& 
\times \idotsint\limits_{ \bd \R_+^n \cap H(1,\mathbf 1_n)} 
\l_{n-1} ([ z_1, \dots,  z_n])^{k} 
\prod\limits_{i=1}^n \I( z_i  \in e_{f_i}^\perp ) 
  \, d z_1  \cdots d z_n  
\\ &=:&
n^{- \frac {k}2}(n-1)^{- \frac {n}2}
\sum_{\bm f \in \{ 1, \dots, n\}_{\neq}^n } \ 
\left( \prod_{i=1}^n \frac 1 {u_i} \right)^{n+k} 
\prod\limits_{i=1}^n u_{f_i} \ \mathcal E_{k,\bm f}, 
\end{eqnarray*}
where $\cE_{k,\bm f}$ is independent of $u$. Together with \eqref{eq:Ekhu} this finishes the proof. 
\end{proof}

When estimating the error term in Section \ref{ssec:error1} we need a second slightly different result which estimates $\cE_0 (h,u)$ for a cube. In that section we assume in \eqref{eq:def-uk-kleiner} that there is a $ k \leq n-1$ such that 
\begin{equation}\label{eq:k-def}
 \frac h{u_1} , \dots , \frac h{u_k}  \leq 1 \mbox{\ \ and\ \  } \frac h{u_{k+1}}, \dots, \frac h {u_n} \geq 1 . 
\end{equation}
Then $H$ meets the coordinate axes in the points $\frac h{u_i} e_i \in [0,1]^n$ for $i=1, \dots, k$, and the other points of intersection are outside of $[0, 1]^n$. 
We set
\begin{eqnarray*}
\cE_0^1 (h, u)
&=& 
\idotsint\limits_{(\bd [0,1]^n  \cap H)^n} 
\prod\limits_{j=1}^n J (T_{x_j},H)^{-1}  
\prod\limits_{f=1}^n \I\left( \left| \{ x_1, \dots, x_n\} \cap e_f^\perp \right| \leq n-1 \right)
 \, d  x_1 \dots d  x_n .
\end{eqnarray*}

\begin{lemma}\label{le:rdsimpl-in-simpl-error}
Let $ k \leq n-1$ be given such that  \eqref{eq:k-def} holds. Then we have
\begin{eqnarray*}
\cE_0^1 (h, u)
 &\leq &
c_{n, k}
h^{-n} \prod_{j=1}^k \left( \frac h{u_j}    \right)^n
\sum_{\bm f \in \{1, \dots, n\}^n}
\prod_{j =1}^k \left( \frac {u_{j}}h \right)^{m_j}
\end{eqnarray*}
with
$m_j = m_j(\bm f)= \sum_{i=1}^{n} \I( f_i=j) \leq n-1$ for $j \leq k$, and $\sum_{i=1}^k m_i \leq n$.
\end{lemma}
\begin{proof}
First we compare the intersection of $H$ with the facet of $[0,1]^n $ in $e_f^\perp$ to the intersection of $H$ with the opposite facet of $[0,1]^n $ in $e_f +e_f^\perp$, $f=1, \dots, n$.
For $i \neq f$ the hyperplane $H$ meets the coordinate axes $\lin \{ e_i\}$ in $e_f^\perp$ in points $\frac h{u_i} e_i$. It meets the shifted coordinate axes $e_f+ \lin \{ e_i\}$ in the opposite facet in the points $e_f+\frac {h-u_f}{u_i} e_i$. Because $u \in S_+^{n-1}$ we have $u_f \geq 0$.
This shows that the facet of $H \cap [0,1]^n$  in $e_f^\perp$ contains the simplex $[\{ \frac h{u_i} e_i \}_{i \leq k,\, i \neq f}]$. The opposite facet contains either the smaller simplex $[\{ \frac {h-u_f}{u_i} e_i \}_{i \leq k,\, i \neq f}]$ if $f \geq k+1$ and $\frac h {u_f} >1$ , and otherwise the intersection is empty, $(H \cap [0,1]^n) \cap (e_f + e_f^\perp) = \emptyset$. 
The simplex $[\{ \frac h{u_i} e_i \}_{i \leq k,\, i \neq f}]$ has volume 
$$ \frac 1{(k-2)!} \frac 1{h ( 1- u_f^2 )^{-\frac 12}} \prod_{i \leq k,\, i \neq f} \frac h{u_i}    $$
for $f \leq k$, and 
$$ \frac 1{(k-1)!} \frac 1{h ( 1- u_f^2 )^{-\frac 12}} \prod_{i \leq k} \frac h{u_i}    $$
for $f \geq k+1$, the volume in the opposite facet clearly is smaller for $f \geq k+1$ or vanishes for $f \leq k$.

We use $J (T_{x},H(h,u))^{-1} = J (T_{x},H(1,u))^{-1}  = \Big( 1- u_f^2 \Big)^{-\frac 12} $ for $x \in e_f^\perp$.
For $f \leq k$ this proves
\begin{eqnarray*}
\int\limits_{([0,1]^n \cap H) \cap e_f^\perp} 
J (T_{x},H)^{-1}  
\, d x
&=& 
\Big( 1- u_f^2 \Big)^{-\frac 12}
\int\limits_{([0,1]^n \cap H) \cap e_f^\perp} 
\, d x
\\ & \leq &
 \frac 1{(k-2)!} (n-k)^{\frac {n-k}2}  \frac 1h \prod_{i \leq k,\, i \neq f} \frac h{u_i} ,
\end{eqnarray*}
since $(n-k)^\frac 12$ is the diameter of the $(n-k)$-dimensional unit cube. In this case there is no simplex in the opposite facet.
Analogously, for $f \geq k+1$
\begin{eqnarray*}
\int\limits_{([0,1]^n \cap H) \cap e_f^\perp} 
J (T_{x},H)^{-1}  
\, d x
& \leq &
\frac 1{(k-1)!} (n-k-1)^{\frac {n-k-1}2}   \frac 1h \prod_{i \leq k} \frac h{u_i} , 
\end{eqnarray*}
and 
\begin{eqnarray*}
\int\limits_{([0,1]^n \cap H) \cap ( e_f + e_f^\perp)} 
J (T_{x},H)^{-1}  
\, d x
& \leq &
 \frac 1{(k-1)!} (n-k-1)^{\frac {n-k-1}2}  \frac 1h \prod_{i \leq k} \frac h{u_i}   .
\end{eqnarray*}
Again we condition on the facets $\bd [0,1]^n  \cap H(1,u)$ from where the random points are chosen. Because of the term 
$\I\left( \left| \{ x_1, \dots, x_n\} \cap e_f^\perp \right| \leq n-1 \right)$, 
it is impossible that all points are contained in one of the facets in $e_f^\perp$. 
Thus for $f \leq k$ we have at  most $n-1$ points in $([0,1]^n \cap H) \cap e_f^\perp$ and no point in $([0,1]^n \cap H) \cap (e_f +e_f^\perp)$ because this set is empty.
For $f \geq k+1$ we have at  most $n-1$ points in $([0,1]^n \cap H) \cap e_f^\perp$ and maybe some additional points in $([0,1]^n \cap H) \cap (e_f +e_f^\perp)$.

Now for $j=1, \dots, n$ there is some $f_j$ such that $ x_j$ is either in the facet $[0,1]^n \cap e_{f_j}^\perp$ or in the opposite facet $[0,1]^n \cap (e_{f_j} + e_{f_j}^\perp)$.
This defines a vector 
$$ \bm  f \in \{1, \dots , n\}^n , $$
and we take into account that for $f \leq k$
$$ m_f= \sum_{j=1}^{n} \I(f_j=f) \leq n-1 . $$
This yields
\begin{eqnarray*}
\cE_0^1 (h, u)
&=& 
\sum_{\bm f \in \{1, \dots, n\}^n}
\ \prod_{j : f_j\leq k} \left( \int\limits_{\bd [0,1]^n  \cap H}  J (T_{x_j},H)^{-1}  
\I\left( x_j \in e_{f_j}^\perp ) \right)
 \, d  x_j \right)
\\ &&
 \hskip1.5cm \times \prod_{j\colon f_j \geq k+1} \left( \int\limits_{\bd [0,1]^n  \cap H}  J (T_{x_j},H)^{-1}  
\I\left( x_j \in e_{f_j}^\perp \cup (e_{f_j} + e_{f_j}^\perp) \right)
 \, d  x_j \right)
\\ &\leq &
c_{n, k}
h^{-n} \prod_{j=1}^k \left( \frac h{u_j}    \right)^n
\sum_{\bm f \in \{1, \dots, n\}^n}
\prod_{j =1}^k \left( \frac {u_{j}}h \right)^{m_j}
\end{eqnarray*}
with
$m_j = m_j(\bm f) = \sum_i \I( f_i=j) \leq n-1$ for $j \leq k$, and $\sum_1^k m_i \leq \sum_1^n m_i =n$.

\end{proof}

\subsection{The crucial substitution}

In the next sections we will end up with integrals over $u \in S_+^{n-1}$ and 
where we split the integrals in the part where $h \geq \min _{1\leq i \leq n} u_i$ and the part where 
$h \leq \min _{1\leq i \leq n} u_i$. Then the following substitution is helpful.

\begin{lemma}\label{le:cruc-subst}
Let $f: (\R_+)^n \to \mathbb{R}$ be an integrable function such that both sides of the following equation are finite. Then 
\begin{eqnarray*}
\int \limits_{S_+^{n-1}} \int\limits_{\R_+}
f\left( \frac h {u_1}, \dots, \frac h{u_n}\right)
h^{-(n+1) } 
\ dh \, du
&=& 
\idotsint\limits_{(\R_+)^n}
f(t_1, \dots, t_n) \prod_{i=1}^{n}t_{i}^{-2}\, 
dt_1 \dots dt_n .
\end{eqnarray*}
\end{lemma}
In particular we will make extensive use of the following version where we use that  
the range of integration $0 \leq h \leq u_i$ for all $i=1, \dots, n$, is equivalent to $t_i \in [0,1]$:
\begin{eqnarray*}
\int \limits_{S_+^{n-1}} \lefteqn{\int\limits_0^{\min_{1 \leq i \leq n} \{ u_i \}}
\left(1- a \frac{\sum u_i}{ \prod u_i}  h^{n-1} \right)^{N-n}   
h^{-(n+1) } 
\prod_{i=1}^n \left(\frac{u_i}h\right)^{-(n+1 + \e) + m_i} \ dh \, du
}&&
\\ &=& 
\int \limits_0^1 \dots  \int\limits_0^1
\left(1- a \sum_i \prod_{j \neq i} t_j \right)^{N-n}
\ \prod_{i=1}^n t_i^{n-1  + \e-m_i} 
dt_1 \dots dt_n,
\end{eqnarray*}
where $\e \in \{0,1\}$, $N$ is the number of chosen points and the $m_i$ are as in Lemma \ref{le:rdsimpl-in-simpl-error}.

\begin{proof}
The goal is to rewrite the integration $dh\, du$ over the set of hyperplanes into an integration with respect to $t_1, \dots, t_n$ where these are the intersections of the hyperplane $H(h,u)$ with the coordinate axis. First, the substitution $r = h^{-1}$ leads to $dh = -r^{-2} dr$. Then we pass from  polar coordinates $(r,u)$ to the Cartesian coordinate system: for $h,r \in \R_+$ and $u\in S^{n-1}_+$ this gives
$$
h^{-(n+1)}\, dh du =r^{n-1}\, dr du =dx_{1}\cdots dx_{n}.
$$
Now we substitute $x_{i}=\frac{1}{t_{i}}$ and take into account that 
$$
h^{-1}=r=\left(\sum_{i=1}^{n}|x_{i}|^{2}\right)^{\frac{1}{2}}
=\left(\sum_{i=1}^{n}\left|\frac{1}{t_{i}}\right|^{2}\right)^{\frac{1}{2}}.
$$
Thus finally we have 
$$
h^{-(n+1)}\, dhdu = \prod_{i=1}^{n}t_{i}^{-2}\, dt_{1}\cdots dt_{n}
$$
with
$h^{-1} u_i= r u_i=x_i=t_i^{-1}$.
\end{proof}

\subsection{The main term}\label{ssec:mainterm}

By \eqref{eq:defI0} and \eqref{eq:defI1}, for $\e \in \{0,1\}$  we have to investigate 
\begin{eqnarray*} 
I_v^{1+\e}
&=& 
a_v^n  \int \limits_{S_+^{n-1}}   \int \limits_0^{\min u_i}  \ 
(1- a_v \l_{n-1}(\bd \R_+^n \cap H_- ))^{N-n}  h^{\e}
\\ && \hskip3.5cm \nonumber  
\times \idotsint\limits_{(\bd \R_+^n \cap H)_{\neq}^n} 
\l_{n-1} ( [x_1, \dots , x_n])^{1+\e}  \prod\limits_{j=1}^n J (T_{x_j},H)^{-1}  
\, d x_1 \dots d x_n \, dh du    
\\ &=& 
a_v^n  \int \limits_{S_+^{n-1}}   \int \limits_0^{\min u_i}  \ 
(1- a_v \l_{n-1}(\bd \R_+^n \cap H_- ))^{N-n}  h^{\e}
\cE_{1+\e} (h, u)
 \, dh du  
\end{eqnarray*}
where we use the notation from (\ref{eq:def_Ek}). Recall that $H=H(h,u)$ meets the coordinate axis in the points $t_i e_i = \frac h {u_i} e_i$, and hence
$$ 
\l_{n-1}(\bd \R_+^n \cap H_- )=
\frac 1{(n-1)!}  \frac{\sum u_i}{\prod u_i} h^{n-1}.
$$
We plug this and the result of Lemma \ref{le:rdsimpl-in-simpl} into $I_v^{1+\e}$, set $m_i= \sum_j \I(f_j=i)$ and obtain
\begin{eqnarray*} 
I_v^{1+\e}
&=& 
n^{- \frac {1+\e}2} (n-1)^{- \frac {n}2} a_v^n
\sum_{\bm f \in \{ 1, \dots, n\}_{\neq}^n } \cE_{1+\e,\bm f}
\\&&
\int \limits_{S_+^{n-1}}   \int \limits_0^{\min u_i}  \ 
\left(1- \frac {a_v}{(n-1)!}  \frac{\sum u_i}{\prod u_i} h^{n-1} \right)^{N-n}  
h^{ - (n +1)}  
\prod_{i=1}^n \left(\frac{u_i}h \right)^{-(n+1+\e)+m_i}  
 \, dh du  .
\end{eqnarray*}
Note that $\sum m_i=n$. 
We use the substitution introduced in Lemma \ref{le:cruc-subst} and use the notation from Lemma \ref{le:main-asymp-exp1} and Lemma \ref{le:main-asymp-exp2}, in particular we use the function $\cJ(\cdot)$ introduced in \eqref{GrossJ}. 
\begin{eqnarray*} 
I_v^{1+\e}
&=& 
n^{- \frac {1+\e}2} (n-1)^{- \frac {n}2} a_v^n (n-1)^{-\e} 
\sum_{\bm f \in \{ 1, \dots, n\}_{\neq}^n } \cE_{1+\e,\bm f}
\\&&
\int \limits_0^1 \dots  \int\limits_0^1
\left(1- \frac {a_v}{(n-1)!}   \sum_i \prod_{j \neq i} t_j \right)^{N-n}
\ \prod_{i=1}^n t_i^{n-1  + \e-m_i} 
dt_1 \dots dt_n
\\ &=&
n^{- \frac {1+\e}2} (n-1)^{- \frac {n}2} a_v^n (n-1)^{-\e} 
\sum_{\bm f \in \{ 1, \dots, n\}_{\neq}^n } \cE_{1+\e,\bm f} \ 
\cJ(\bm m - (1+\e) \bm 1)
\end{eqnarray*}
with $\bm m=(m_1, \dots, m_n)$. 
In the case $\e=1$, 
\begin{eqnarray*} 
I_v^2
&=& 
n^{- \frac {2}2} (n-1)^{- \frac {n}2} a_v^n (n-1)^{-1} 
\sum_{\bm f \in \{ 1, \dots, n\}_{\neq}^n } \cE_{2,\bm f} \ 
\cJ(\bm m - 2\,  \cdot \bm 1),
\end{eqnarray*}
and Lemma \ref{le:main-asymp-exp1} (with $L=-n$) implies with a constant $c_{\bm m , n}$  that depends on $\bm m$ and $n$,
\begin{eqnarray}\label{eq:I_v^2} 
I_v^2
&=& \nonumber
c_n  a_v^{-\frac {n}{n-1} }
\sum_{\bm f } \cE_{2, \bm f} \ 
 c_{\bm m , n} \ N ^{-n-\frac{n}{n-1}} (1+ O(N^{- \frac 1{(n-1)(n-2)}}))
\\ &=&
c_n  a_v^{-\frac {n}{n-1} }
\ N ^{-n-\frac{n}{n-1}} (1+ O(N^{- \frac 1{(n-1)(n-2)}})) , 
\end{eqnarray}
where the implicit constant in $O(\cdot)$ may depend on $a_v$. 
Because $c_{\bm m, n}>0$, all terms with $\bm f \in \{ 1, \dots, n\}_{\neq}^n $ contribute.
Geometrically this means that \emph{the contribution for the volume difference comes from all facets of $P_N$.}

In the case $\e=0$ the asymptotic results from Lemma \ref{le:main-asymp-exp2} (with $L=0$) give 
\begin{eqnarray}\label{eq:I_v^1} 
I_v^1
&=& \nonumber
n^{- \frac {1}2} (n-1)^{- \frac {n}2} a_v^n 
\sum_{\bm f \in \{ 1, \dots, n\}_{\neq}^n } \cE_{1,\bm f} \ 
\cJ(\bm m - \bm 1)
\\ &=& \nonumber
c_n \sum_{\bm f:\, \sharp \{m_i>0\}=2 } \cE_{1,\bm f} \ d_{\bm m,n} N^{-n} (\ln N)^{n-2} (1+O((\ln N)^{-1}))
+
c_n \sum_{\bm f:\, \sharp \{m_i>0\}\geq 3 } O(N^{-n} (\ln N)^{n-3})
\\ &=&
c_n N^{-n} (\ln N)^{n-2} (1+O((\ln N)^{-1}))
\end{eqnarray}
where only those terms contribute for which $f_i$ is concentrated on two values, and where the implicit constant in $O(\cdot)$ may depend on $a_v$. 
We can apply Lemma \ref{le:main-asymp-exp2} as (\ref{alpha-klein}) holds.

Geometrically this implies that \emph{the main contribution comes from that facets of $P_N$ whose vertices are on precisely two facets of $P$.}

\subsection{The error of the first kind}\label{ssec:error1}

Denote by $\diam (K)$ the diameter of a convex set $K$.  By \eqref{E-Term} and \eqref{eq:defE0}, for the error term we have to estimate
\begin{eqnarray*}
E_v^{1+\e}
&\leq&
(2\overline a)^n \int \limits_{S_+^{n-1}}   \int \limits_{\min u_i}^{\diam (A_v P)} 
 (1- \underline a \l_{n-1}(\bd  A_v P \cap H_-))^{N} h^{\e} 
\\ && \hskip3cm 
\idotsint\limits_{(\bd A_v  P \cap H)_{\neq}^n} 
 \l_{n-1} ( [x_1, \dots , x_n])^{1+\e}  \prod\limits_{j=1}^n J (T_{x_j},H)^{-1}  
 \, d  x_1 \dots d  x_n \, dh du   .
\\ & \leq &
(2\overline a)^n  \int \limits_{S_+^{n-1}}   \int \limits_{\min u_i}^{\diam (A_v P)} 
 (1- \underline a \l_{n-1}(\bd  A_v P \cap H_-))^{N}  h^{\e}
\\ && \hskip3cm 
 \l_{n-1}(A_v P \cap H)^{1+\e}  \idotsint\limits_{(\bd A_v  P \cap H)_{\neq}^n} 
 \prod\limits_{j=1}^n J (T_{x_j},H)^{-1}  
 \, d  x_1 \dots d  x_n \, dh du   .
\end{eqnarray*}
for $\e=0,1$.
Recall that the hyperplane $H=H(h,u)$ meets the coordinate axes in the points $\frac h{u_i} e_i$. 
Hence the halfspace $H_-$ contains at least one unit vector since $h \geq \min u_i$. W.l.o.g. we multiply by $n \choose k$, assume that it contains $e_{k+1}, \dots, e_n$, and thus the points of intersection satisfy 
\begin{equation}\label{eq:def-uk-kleiner}
\frac h{u_1} , \dots , \frac h{u_k}  \leq 1 \mbox{\ \ and\ \  } \frac h{u_{k+1}}, \dots, \frac h {u_n} \geq 1 
\end{equation}
with some $0 \leq k \leq n-1$.
Then the convex hull of $\frac h{u_1} e_1, \dots , \frac h{u_k} e_k, e_{k+1}, \dots, e_n$ is contained in $A_v P \cap H_-$ and we estimate
\begin{eqnarray}\label{eq:calculate-SAvP}
\l_{n-1}( \bd  A_v P \cap H_-) 
&\geq & \nonumber
\frac{1}{(n-1)!}  \sum\limits_{j=1}^n \prod\limits_{i \neq j} \min(1, \frac h{u_i}) 
\\ &\geq &
\frac{1}{(n-1)!}  \sum\limits_{j=1}^k \prod\limits_{i \leq k, i \neq j} \frac h{u_i} .
\end{eqnarray}
For $k=0,1$ we have 
$\l_{n-1}( \bd  A_v P \cap H_-) 
\geq 
\frac{1}{(n-1)!} $
and thus $E_\e = O(e^{- \frac{\underline{a}}{(n-1)!} N})$, so serious estimates are only necessary in the cases $2 \leq k \leq n-1$.
Next we use that $A_v P \subset [0, \t]^n$ for all $A_v$ and for some $\t >0$. Thus
\begin{equation}\label{eq:calculate-SAvPH}
 \l_{n-1}(A_v P \cap H)  \leq \l_{n-1}([0, \t]^n \cap H) \leq c_n h^{-1} \prod_{i=1}^k \frac h{u_i} \ \t^{n-k}  
\end{equation}
because $H$ meets the first $k$ coordinate axes in $\frac h{u_1}, \dots, \frac h{u_k}$.
This gives
\begin{eqnarray*}
E_v^{1+\e}
& \leq &
(2\overline a)^n 
c_n \sum _{k=0}^{n-1} {n \choose k} \t^{(n-k)(1+\e)} 
\int \limits_{S_+^{n-1}}   
\int \limits_{\begin{array}{c}
\scriptstyle  h \leq {u_1} , \dots , {u_k}  \\ \scriptstyle h \geq {u_{k+1}}, \dots, {u_n} 
\end{array} }
\left(1- \frac{\underline a}{(n-1)!}  \sum\limits_{j=1}^k \prod\limits_{i \leq k, i \neq j} \frac h{u_i}  \right)^{N}  
h^{-1} 
\\ && \hskip3cm 
\times \prod_1^k  \left( \frac h{u_i} \right)^{1+ \e}
\idotsint\limits_{(\bd A_v  P \cap H)_{\neq}^n} 
\prod\limits_{j=1}^n J (T_{x_j},H)^{-1}  
 \, d  x_1 \dots d  x_n \, dh du  .
\end{eqnarray*}
Now we deal with the inner integration with respect to $x_1, \dots, x_n$. We want to replace $\bd A_v  P \cap H$ by $\bd [0,1]^n \cap H$. The main point here is to estimate $J(T_x, H)^{-1}$ for $x \notin \bd \R_+^n$. 

In general we have $J (T_{x}, H ) \in [0,1]$ by definition. Recall that $x \in H$. The critical equality $J (T_{x},H) =0$ can occur only if $T_x=H$, thus if $H$ is a supporting hyperplane $H(h_{A_v P}(u),u)$ or $H(h_{A_v P}(-u),-u)$ to ${A_v P}$. 
Since $u \in S_+^{n-1}$, in the second case we have $h_{A_v P}(-u)=0$ and $x \in \bd \R_+^n$.

To exclude the first case we assume that $\l_{n-1}(\bd  A_v P \cap H_-) \leq \frac 12 $. In this case $H_-$ cannot contain the point $n^{- \frac 1{n-1}} (1, \dots, 1)^T$ since otherwise $\bd A_v P \cap H_-$ would contain $\bd A_v P \cap n^{- \frac 1{n-1}} [0,1]^n$ (recall that $u \in S_+^{n-1}$) and this part has surface are 1. 
Now we claim that there is a constant $c_{A_v P} >0$ such that 
\begin{equation}\label{def:cAvP}
J(T_x, H) \geq c_{A_v P} 
\ \text{ if }\  
\l_{n-1}(\bd  A_v P \cap H_-) \leq \frac 12
\ \text{ and }\ 
x \in \bd A_v P \setminus \bd \R_+^n.
\end{equation}
If such a positive constant would not exist then (by the compactness of $\bd A_v P$) there would be a convergent sequence  $(x_k, H_k) \to (x_0, H_0)$ with $ J (T_{x_k},H_k)  \to 0$, where $x_k \in H_k$ yields $x_0 \in H_0=H(h_0, u_0)$, $u_0 \in S_+^{n-1}$.
But in this case also 
$$ J(T_{x_{k}}, H_0) \to 0  $$
and $H_0$ is a supporting hyperplane at $x_0$.
Since $u_0 \in S_+^{n-1}$ this leads to two cases. The first case is that 
$$
x_0 \in H_0= H(h_{A_v P}(-u_0),-u_0)  ,\   x_0 \in \bd A_v P \cap \bd \R^n_+
, $$
but  $x_{k}$ is not in $\bd A_v P \cap \bd \R^n_+$ and thus contained in some other facet of $A_v P$. This implies $J(T_{x_{k}}, H_0) \nrightarrow 0$ as $x_k \to x_0$, which is impossible.
The second case is that  $x_0$ is contained  in $H_0=H(h_{A_v P}(u_0),u_0)$ where $(1,\dots, 1)^T \in H_{0-} $.  Since this point is in $A_v P$, but all $H_{k -}$ do not contain $n^{- \frac 1{n-1}} (1, \dots, 1)^T$ this  again contradicts the convergence $H_k \to H_0$.
Hence such a sequence $x_k$ cannot exist, and \eqref{def:cAvP} holds with some constant $c_{A_v P}>0$.
Thus from now on we assume that $\l_{n-1}(\bd  A_v P \cap H_-) \leq \frac 12$, take into account an error term of order 
\begin{equation}\label{Irre}
(1- \frac{1}2 \underline a)^{N}=e^{-cN},
\end{equation}
and obtain by \eqref{def:cAvP} that 
\begin{equation}\label{eq:J-1-estimate}
\int\limits_{(\bd A_v  P \setminus \bd \R^n_+)\cap H } 
J (T_{x},H)^{-1} \, d  x 
\leq
c_{A_v P}^{-1} \,
\l_{n-2} (\bd [0,\t]^n\cap H )
=
c_{A_v P}^{-1}
\int\limits_{\bd [0,\t]^n \cap H } dx 
\end{equation}
because $A_v P$ is contained in the larger cube $[0,\t]^n$.

In the following we denote by $F_c$ the union of the facets of $A_v P$ contained in $\bd \R_+^n$, and by $F_{0}$ the union of the remaining facets which cover $\bd A_v  P \setminus \bd \R^n_+$, $\bd A_v P  = F_c \cup F_0 $. Then
\begin{eqnarray*}
\idotsint\limits_{(\bd A_v  P \cap H)_{\neq}^n} 
\prod\limits_{j=1}^n J (T_{x_j},H)^{-1}  
 \, d  x_1 \dots d  x_n 
& \leq &
\idotsint\limits_{(F_c \cap H)_{\neq}^n} 
\prod\limits_{j=1}^n J (T_{x_j},H)^{-1}  
 \, d  x_1 \dots d  x_n 
\\&&  +
 \sum_{k=1}^n {n \choose k} 
\idotsint\limits_{(F_0 \cap H)^{k} \times (F_c \cap H)^{n-k}} 
\prod\limits_{j=1}^n J (T_{x_j},H)^{-1}  
 \ d  x_1 \dots d  x_n .
\end{eqnarray*}
Because of \eqref{eq:J-1-estimate} and using $F_c \subset \bd [0,\t]^n $, we obtain the upper bounds
$$
\idotsint\limits_{(F_0 \cap H)^{k} } 
\prod\limits_{j=1}^k J (T_{x_j},H)^{-1}  
 \ d  x_1 \dots d  x_k 
\leq
c_{A_v P}^{-n} 
\idotsint\limits_{(\bd [0,\t]^n \cap H)^k } d x_1 \dots d x_k 
$$
and
$$
\idotsint\limits_{(F_c \cap H)^{n-k}_{\neq}} 
\prod\limits_{j=k+1}^n J (T_{x_j},H)^{-1}  
 \ d  x_{k+1} \dots d  x_n 
\leq 
\idotsint\limits_{(\bd [0,\t]^n \cap H)^{n-k}_{\neq}} 
\prod\limits_{j=k+1}^n J (T_{x_j},H)^{-1}  
 \ d  x_{k+1} \dots d  x_n 
$$
where $(F_c \cap H)^{n-k}_{\neq}=(F_c \cap H)^{n-k}$ for $k \geq 1$. Combining these we get for $k \geq 1$
\begin{eqnarray*}
\idotsint\limits_{(F_0 \cap H)^{k} \times (F_c \cap H)^{n-k}} 
\prod\limits_{j=1}^n J (T_{x_j},H)^{-1}  
 \ d  x_1 \dots d  x_n 
&\leq &
c_{A_v P}^{-n} 
\idotsint\limits_{(\bd [0,\t]^n \cap H)^n} 
\prod\limits_{j=k+1}^n J (T_{x_j},H)^{-1}  
 \ d  x_{1} \dots d  x_n .
\end{eqnarray*}
Observe that for the $(n-1)$-dimensional polytope $[0,\t]^n \cap H $ the area of each $(n-2)$-dimensional facet is bounded by the sum of the areas of all other facets. Hence excluding a facet from the range of integration of the inner integral with respect to $x_1$ can be compensated by a factor 2,
$$
\int\limits_{\bd [0,\t]^n \cap H}  \ d  x_{1} 
\leq
2 \int\limits_{\bd [0,\t]^n \cap H} \prod\limits_{f=1}^n \I(|\{x_1, \dots, x_n\}\cap e_f^\perp | \leq n-1))\ d  x_{1} .
$$
Since $J (T_{x_j},H)$ is always less or equal one, this yields
\begin{eqnarray*}
\idotsint\limits_{(\bd A_v  P \cap H)_{\neq}^n} 
\lefteqn{
\prod\limits_{j=1}^n J (T_{x_j},H)^{-1}  
 \, d  x_1 \dots d  x_n 
}& &
\\ &\leq &
2 c_{A_v P}^{-n} 
\sum_{k=0}^n {n \choose k} 
\idotsint\limits_{{(\bd [0,\t]^n \cap H)^n} } 
\prod\limits_{j=1}^n J (T_{x_j},H)^{-1} 
\prod\limits_{f=1}^n \I(|\{x_1, \dots, x_n\}\cap e_f^\perp | \leq n-1))
 \ d  x_1 \dots d  x_n 
\\ & \leq &
2^{n+1}  c_{A_v P}^{-n}  
\idotsint\limits_{{(\bd [0,\t]^n \cap H)^n} } 
\prod\limits_{j=1}^n J (T_{x_j},H)^{-1}  
\prod\limits_{f=1}^n \I(|\{x_1, \dots, x_n\}\cap e_f^\perp | \leq n-1))
 \ d  x_1 \dots d  x_n.
\end{eqnarray*}
Substituting $x_i$ by $\t x_i$ we obtain
$$ 
\idotsint\limits_{(\bd A_v  P \cap H)_{\neq}^n} 
\prod\limits_{j=1}^n J (T_{x_j},H)^{-1}  
 \, d  x_1 \dots d  x_n 
\leq 
2^n \t^{n(n-2) } c_{A_v P}^{-n}  
\cE_0^1 \left( \frac h \t, u\right) 
$$
with $\cE_0^1 (\cdot)$ defined in front of Lemma \ref{le:rdsimpl-in-simpl-error}. We make use of Lemma \ref{le:rdsimpl-in-simpl-error}  for $\cE_0^1$
and the error term (\ref{Irre}): for  $m_i= \sum_j \I(f_j=i)$ we get
\begin{eqnarray*}
E_v^{1+\e}
& \leq &
c_{v,P}  \sum _{k=0}^{n-1} 
\int \limits_{S_+^{n-1}}   
\int \limits_{ \begin{array}{c}
\scriptstyle  h \leq {u_1} , \dots , {u_k}  \\ \scriptstyle h \geq {u_{k+1}}, \dots, {u_n} 
\end{array} }
\left(1- \frac{\underline a}{(n-1)!}  \sum\limits_{j=1}^k \prod\limits_{i \leq k, i \neq j} \frac h{u_i}  \right)^{N}  
h^{ -(n+1) }  
\\ && \hskip3.5cm 
\prod_{j=1}^k \left( \frac h{u_j}\right)^{n+1+\e}  \left( \sum_{\bm f \in \{1, \dots, n\}^n}
\prod_{j =1}^k \left( \frac {u_{j}}h \right)^{m_j}
 \right)
\, dh du  \ + O(e^{-cN}) 
\end{eqnarray*}
with 
$m_i \leq n-1$ for $j \leq k$, $\sum_1^k m_i \leq n$, and where $c_{v,P}$ depends on $n, c_{A_v P}, \max c_{n,k}$ and $\t$.
Next we use the substitution from Lemma \ref{le:cruc-subst}:
\begin{eqnarray*}
E_v^{1+\e}
& \leq &
c_{n, A_v P}   \sum _{k=0}^{n-1} 
\sum_{\bm f \in \{1, \dots, n\}^n} 
\idotsint \limits_{\begin{array}{c}
\scriptstyle  t_1, \dots t_k \leq 1  \\ \scriptstyle  t_{k+1}, \dots, t_n \geq 1
\end{array} }
\left(1- \frac{\underline a}{(n-1)!}  \sum\limits_{j=1}^k \prod\limits_{i \leq k, i \neq j} t_i  \right)^{N}  
\\ && \hskip5cm 
\prod_{i=1}^k t_i^{n-1+\e -m_i}
\prod_{i=k+1}^n t_i^{-2}
\, dt_1 \dots dt_n  \ + O(e^{-cN})   .
\end{eqnarray*}
The integrations with respect to $t_{k+1}, \dots , t_n$ are immediate since the only terms occurring are $t_i^{-2}$, and we have 
\begin{eqnarray*}
E_v^{1+\e}
& \leq &
c_{n, A_v P}  \sum _{k=0}^{n-1} 
\sum_{\bm f \in \{1, \dots, n\}^n} 
\int\limits_0^1 \dots \int\limits_0^1
\left(1- \frac{\underline a}{(n-1)!}  \sum\limits_{j=1}^k \prod\limits_{i \leq k, i \neq j} t_i  \right)^{N}  
\\ && \hskip5cm 
\prod_{i=1}^k t_i^{k-2 - (m_i- (n-k+1+\e))}
\, dt_1 \dots dt_k \ + O(e^{-cN})   .
\end{eqnarray*}
with $0 \leq m_i \leq n-1$. We set $l_i=m_i- (n-k+1+\e))$.
To apply Lemma \ref{le:main-asymp-exp2} in the case $\e=0$ we have to check that there are $i \neq j$ with $l_i, l_j > \frac L{k-1} -1$. Set $M=\sum_1^k m_i \leq n$. We have 
$$ 
l_i -\frac L{k-1} +1 =
m_i -(n-k+1) - \frac {\sum_{j=1}^k (m_j- (n-k+1))}{k-1} + 1 =
m_i  + \frac {n-M}{k-1} \geq 0
$$
and equality holds only if $M=n$ and $m_i=0$. But $M=n $ and $m_i \leq n-1$ imply that there are at least two different indices $i, j$ with $m_i >0$. Hence we may apply Lemma \ref{le:main-asymp-exp2} (and if $m_i \geq 1$ for all $i$ even Lemma \ref{le:main-asymp-exp1}) which tells us that the integral is bounded by 
$$
O(N^{-k+\frac L{k-1}} (\ln N)^{k-2}) =
O(N^{\frac {M-nk}{k-1}} (\ln N)^{n-3})  =
O(N^{-n} (\ln N)^{n-3})  .
$$
This finally proves 
\begin{equation}\label{eq:E_v^1}
 E_v^1 = 
O(N^{-n} (\ln N)^{n-3})  .
\end{equation}
In the case $\e=1$ we have $l_i=m_i- (n-k+2)$, and with $M=\sum_1^k m_i \leq n$ this gives
$$ 
l_i -\frac L{k-1} +1 =
m_i -(n-k+2) - \frac {\sum_{j=1}^k (m_j- (n-k+2))}{k-1} + 1 =
m_i + \frac {n+1-M}{k-1}  > 0
$$
since $M \leq n$. Thus the integral is of order
$
O(N^{-k+\frac L{k-1}})= 
O(N^{\frac {M-k(n+1)}{k-1}} )= 
O(N^{-n -\frac {n-1}{n-2}} )
$
and 
\begin{equation}\label{eq:E_v^2}
E_v^2 
=
O(N^{-n -\frac {n-1}{n-2}} ) .
\end{equation}

\subsection{The error of the second kind}\label{ssec:error2}

Here we have to evaluate the following estimate of $\E V_n(D_N)$ of (\ref{eq:F-estimate}), 
$$
f_1(P) \, \overline d  \int\limits_0^\t \dots \int\limits_0^\t   \left( 1- \underline a\l_{n-1} ( \bd \R_+^n \cap [0, \min(1,x_1) e_1, \dots, \min(1,x_{n-1}) e_{n-1}, e_n  ])   \right)^{N} dx_1 \dots dx_n 
.
$$
The integration with respect to $x_n$ is immediate. We may  assume without loss of generality that for $k = 0, \dots, n-1$ precisely $k$ of the coordinates of $x$ are bounded by $1$,
$$ x_1, \dots, x_k \leq 1,\ x_{k+1} , \dots, x_{n-1} \geq 1 .$$
For $k=0,1$ we have 
$$ \int\limits_0^\t \dots \int\limits_0^\t 
\left( 1- \frac{\underline a} {(n-1)!}   \right)^{N} dx_1 \dots dx_{n-1} 
= O(e^{- \frac{\underline{a}}{(n-1)!} N}) . $$
So we assume $2 \leq k \leq n-1$.
Then the volume of the boundary of the simplex is given by 
\begin{align*}
\l_{n-1} ( \bd \R_+^n \cap [0, \min(1,x_1) e_1, \dots, \min(1,x_{n-1}) e_{n-1}, e_n  ])
=&
\sum_{j=1}^n \frac 1{(n-1)!} \prod_{i \leq n, i \neq j} \min(1,x_i)
\\ \geq &
\sum_{j=1}^k \frac 1{(n-1)!} \prod_{i \leq k, i \neq j} x_i .
\end{align*}
Therefore we obtain
\begin{eqnarray*}
\int\limits_0^\t \dots \int\limits_0^\t 
\lefteqn{
\left( 1- \frac {\underline a }{(n-1)!} \sum_{j=1}^k  \prod_{i \leq k, i \neq j} x_i  \right)^{N} dx_1 \dots dx_{n-1} 
} &&
\\ & \leq &
\t^n \int\limits_0^1 \dots \int\limits_0^1 
\left( 1- \frac {\underline a \, \t^{k-1}}{(n-1)!} \sum_{j=1}^k  \prod_{i \leq k, i \neq j} t_i  \right)^{N-k} dt_1 \dots dt_{n-1} 
\\ & \leq &
\t^n \left( \frac {\underline a \, \t^{k-1}}{(n-1)!}\right)^{-\frac{k}{k-1} }  (k-1)^{-1} 
\Gamma \left(\frac {1}{k-1}  \right)^k 
\ N ^{-\frac{k}{k-1} } (1+ o(1))
\end{eqnarray*}
where we used Lemma \ref{le:main-asymp-exp1} with $l_i=k-2$, $L= k(k-2)$ which implies 
$l_i 
> \frac L{k-1}-1
= k-2 - \frac{1}{k-1}
$.
As $ k \leq n-1$ we obtain
\begin{equation*}
\E V_n(D_N)
= O(N^{- \frac{n-1}{n-2}})   .
\end{equation*}

\vskip0.5cm
\begin{appendix}
 
\centerline {\Large \bf Appendix: Some Asymptotic Expansions} 
 
\section{A useful substitution} \label{sec:useful-sub}

Let $\mathcal S_{n}$ is the set of all permutations of $\{1,\dots,n\}$. We start with the following observation.

\begin{lemma}\label{SetIntegration1}
Let $f:(0,\infty)^{n}\to(0,\infty)^{n}$ be defined by
$$
f_{j}(x)=\prod_{i\ne j}x_{i}
\hskip 20mm j=1,\dots,n.
$$
(i) The inverse function to $f$ is $g:(0,\infty)^{n}\to(0,\infty)^{n}$ given by
$$
g_{i}(x)=\frac{1}{x_{i}}\left(\prod_{k=1}^{n}x_{k}\right)^{\frac{1}{n-1}}.
$$
(ii) $f$ maps the open set $(0,1)^{n}$ bijectively onto
\begin{equation}\label{SetIntegration1-1}
\left\{y\in(0,1)^{n}\left|\forall i=1,\dots,n: \prod_{k=1}^{n}y_{k}< y_{i}^{n-1}\right.\right\}.
\end{equation}
(iii) The set
\begin{equation}\label{SetIntegration1-2}
\left\{x\in(0,\beta)^{n}\left|\forall i=1,\dots,n: \prod_{k=1}^{n}x_{k}<\beta\cdot x_{i}^{n-1}\right.\right\}
\end{equation}
equals
\begin{equation}\label{SetIntegration1-3}
\bigcup_{\pi\in\mathcal S_{n}}\{(x_{\pi(1)},\dots,x_{\pi(n)})|x\in M\},
\end{equation}
where $M$ is the set of all $x\in(0,\infty)^{n}$ with
$x_{n}\leq x_{n-1} \leq \cdots \leq x_{1}$ and 
\begin{eqnarray}\label{eq:IntArea1-1}
\beta\cdot x_{3}&>& x_{1}\cdot x_{2} \nonumber \\
\beta\cdot x_{4}^{2}&>& x_{1}\cdot x_{2}\cdot x_{3} 
\nonumber \\
&\vdots&   \\
\beta\cdot x_{n}^{n-2}&>& x_{1}\cdots x_{n-1} .
\nonumber
\end{eqnarray}
\end{lemma}
\vskip 3mm

\begin{proof}
(i) For all $j=1,\dots,n$
$$
f_{j}(g(x))
=\prod_{i\ne j}g_{i}(x)
=\prod_{i\ne j}\left(\frac{1}{x_{i}}\left(\prod_{k=1}^{n}x_{k}\right)^{\frac{1}{n-1}}\right)
=x_{j}
$$
and for all $i=1,\dots,n$
$$
g_{i}(f(x))
=g_{i}\left(\prod_{k\ne 1}x_{k},\dots,\prod_{k\ne n}x_{k}\right)
=\left(\prod_{k\ne i}x_{k}\right)^{-1}\left(\prod_{j=1}^{n}\prod_{k\ne j}x_{k}\right)^{\frac{1}{n-1}}
=x_{i}.
$$
(ii) We show that $f$ maps an element $x \in (0,1)^{n}$ to an element of the set defined by 
(\ref{SetIntegration1-1}). Indeed, for all $x\in(0,1)$ we have
$$
\prod_{j\ne i}x_{j}\in(0,1).
$$
Moreover,
$$
\prod_{j=1}^{n}f_{j}(x)
=\prod_{j=1}^{n}\prod_{k\ne j}x_{k}
=\left(\prod_{j=1}^{n}x_{j}\right)^{n-1}.
$$
Since for all $i=1,\dots,n$ we have $x_{i}\in(0,1)$ we get for all
$i=1,\dots,n$
$$
\prod_{j=1}^{n}f_{j}(x)<\left(\prod_{j\ne i}x_{j}\right)^{n-1}=f_{i}(x)^{n-1}.
$$
Thus $f$ maps $(0,\infty)^{n}$ into the set defined by  (\ref{SetIntegration1-1}).

Now we show that $g$ maps an element $y$ of the set defined by (\ref{SetIntegration1-1}) to
an element of $(0,1)^{n}$. Since $\prod_{k=1}^{n}y_{k}< y_{i}^{n-1}$
$$
g_{i}(y)=\frac{1}{y_{i}}\left(\prod_{k=1}^{n}y_{k}\right)^{\frac{1}{n-1}}<1.
$$
(iii) We show that the set defined by $(\ref{SetIntegration1-3})$ contains the set defined by 
(\ref{SetIntegration1-2}). Let $x$ be an element of the set defined by (\ref{SetIntegration1-2}).
There is a permutation $\pi$ such that
$$
x_{\pi(n)} \leq x_{\pi(n-1)} \leq \cdots \leq x_{\pi(1)}
$$
and for all $i=1,\dots,n$
\begin{equation}\label{SetIntegration1-4}
\prod_{k=1}^{n}x_{\pi(k)}<  \beta x_{\pi(i)}^{n-1}.
\end{equation}
We prove by induction that $(x_{\pi(1)},\dots,x_{\pi(n)})\in M$.
The last inequality of (\ref{SetIntegration1-4}) follows from (\ref{SetIntegration1-2}) for $i=n$.
Suppose now that we have verified the last $k$ inequalities,
i.e.
\begin{eqnarray*}
\beta x_{\pi(n)}^{n-2}&>& x_{\pi(1)}\cdots x_{\pi(n-2)}\cdot x_{\pi(n-1)}  \\
\beta x_{\pi(n-1)}^{n-3}&>& x_{\pi(1)}\cdots x_{\pi(n-2)} \\
&\vdots&   \\
\beta x_{\pi(n-k+1)}^{n-k-1}&>& x_{\pi(1)}\cdots x_{\pi(n-k)} .
\end{eqnarray*}
By (\ref{eq:IntArea1-1}),
$$
\beta x_{\pi(n-k)}^{n-1}
>
\prod_{j=1}^{n}x_{\pi(j)} .
$$
We substitute for $x_{n}, \dots , x_{n-k+1}$ using the above inequalities already 
obtained.  
\begin{eqnarray*}
x_{\pi(n-k)}
&>&
\left( \frac 1 {\beta } \prod_{j=1}^{n}x_{\pi(j)} \right)^{\frac 1{n-1}}
 \\ &>&
\left(\left( \frac 1 {\beta } \right)^{1+\frac 1 {n-2}}  \prod_{j=1}^{n-1}x_{\pi(j)}^{1+\frac 1 
{n-2}} \  \right)^{\frac 1{n-1}}
=
\left( \frac{1}{\beta}  \prod_{j=1}^{n-1}x_{\pi(j)} \right) ^{\frac 1{n-2}} \ 
 \\ &>&
\left(\left( \frac 1 {\beta } \right)^{1+\frac 1 {n-3}}  \prod_{j=1}^{n-2}x_{j}^{1+\frac 1 
{n-3}} \  \right)^{\frac 1{n-2}}
=
\left( \frac{1}{\beta}  \prod_{j=1}^{n-2}x_{\pi(j)} \right) ^{\frac 1{n-3}} \ 
\\ & \vdots &
 \\ &>&
\left(\left( \frac 1 {\beta} \right)^{1+\frac 1 {n-k-1}}  \prod_{j=1}^{n-k}x_{\pi(j)}^{1+\frac 1 
{n-k-1}} \  \right)^{\frac 1{n-k}}
=
\left( \frac{1}{\beta}  \prod_{j=1}^{n-k}x_{\pi(j)} \right) ^{\frac 1{n-k-1}} \ 
\end{eqnarray*}
or, equivalently,
$$
x_{\pi(n-k)}^{n-k-2}
>\frac{1}{\beta}  \prod_{j=1}^{n-k-1}x_{\pi(j)}, 
$$
as long as $n-k-2 \geq 1$. Thus the last inequality is $x_{3}>\frac{1}{\beta} \cdot x_{1}\cdot x_{2}$.

Now we show that (\ref{SetIntegration1-3}) is contained in (\ref{SetIntegration1-2}).
It is enough to show that $M$ is a subset of (\ref{SetIntegration1-2}).

Let $x\in M$. By the the last inequality of (\ref{eq:IntArea1-1})
$$
x_{1}\cdot x_{2}\cdots x_{n}<\beta x_{n}^{n-1}.
$$
Since $x_{n}<x_{n-1}<\cdots<x_{1}$ we get for all $i=1,\dots,n$
$$
x_{1}\cdot x_{2}\cdots x_{n}<\beta x_{n}^{n-1}< \beta x_{i}^{n-1}.
$$
\end{proof}

Recall the definition (\ref{GrossJ})
$$
\cJ (\bm l)
=
\int \limits_0^1 \dots  \int\limits_0^1
\left(1-\a \sum_{i=1}^{n} \prod_{j \neq i} t_j \right)^{N-n}
\ \prod_{i=1}^n t_i^{n-2 - l_i} 
dt_1 \dots dt_n
.
$$

\begin{lemma}\label{le:substitute1}
Let $\alpha>0$, and $\bm l=(l_1, \dots, l_n)$, $L=\sum_1^n l_i $, with  $ l_i <n-1 $ for all $i=1,\dots,n$. Then we have
$$
\cJ (\bm l)
= 
\left(\frac1 {\a (N-n)}\right)^{n- \frac{L}{n-1}  }
\frac{1}{n-1} \!\!\!
 \underbrace{
 \int \limits_0^{\a (N-n)} \dots 
\int\limits_0^{\a (N-n)}
}_{\forall i:\ \prod_1^{n}  s_j^{\frac 1{n-1}} \leq(\a (N-n))^{\frac{1}{n-1}} s_i} \!
\left(1- \frac{\sum_{i=1}^{n} s_i }{N-n}  \right)^{N-n}  
\prod_{i=1}^n s_i^{l_{i} - \frac {L}{n-1} }  ds_n \dots ds_1  .
$$
\end{lemma}

\begin{proof}
By the assumption $ l_i <n-1 $ for all $i=1,\dots,n$ the integral
is finite.

We use the transformation of Lemma \ref{SetIntegration1}: 
For $i,j=1,\dots,n$ 
\begin{equation}\label{eq:Int1Int2-1}
v_j= \prod_{i \neq j} t_i 
\hskip 20mm\mbox{and}\hskip 20mm
t_i 
= \frac{\left(\prod_{k=1}^{n} v_k\right)^{\frac 1{n-1}}}{v_i}.
\end{equation}
The partial derivatives of $t$ with respect to $v$ are for $i\ne j$
$$
\frac{d t_i}{dv_j} 
= \frac 1{n-1} \frac{(\prod_{k=1}^{n} v_k)^{\frac 1{n-1}}}{v_jv_i} 
$$
and for $i=j$
$$
\frac{d t_i}{dv_i} 
= \left(\frac 1{n-1} -1\right)  \frac{(\prod_{k=1}^{n} v_k)^{\frac 1{n-1}}}{v_i^2} .
$$
This allows the computation of the Jacobian
\begin{eqnarray*}
J
&=&
\det\left(\frac{\partial t_{i}}{\partial v_{j}}
\right)_{i,j=1}^{n}   
\\ &=&
\left(\prod_{k=1}^{n} v_k\right)^{\frac {n}{n-1}}\det\left(
\begin{array}{ccccc}
(-1+\frac{1}{n-1})\frac{1}{v_{1}^{2}}
&\frac 1{n-1} \frac{1}{v_1v_2}
&\frac 1{n-1} \frac{1}{v_1v_3}
&\cdots&\frac 1{n-1} \frac{1}{v_1v_n}   \\
\frac 1{n-1} \frac{1}{v_1v_2}
&(-1+\frac{1}{n-1})\frac{1}{v_{2}^{2}}
&\frac 1{n-1} \frac{1}{v_2v_3}
&\cdots&
\frac 1{n-1} \frac{1}{v_2v_n}  \\
\vdots &\vdots  &\vdots &\vdots  &\vdots   \\
\frac 1{n-1} \frac{1}{v_1v_n}
&\frac 1{n-1} \frac{1}{v_2v_n}
&\frac 1{n-1} \frac{1}{v_3v_n}
&\cdots
&(-1+\frac{1}{n-1})\frac{1}{v_{n}^{2}}
\end{array}\right)  
\\ &=&
(n-1)^{-n}\left(\prod_{k=1}^{n} v_k\right)^{\frac {n}{n-1}-2}\det\left(
\begin{array}{ccccc}
2-n
&1
&1 
&\cdots&1  \\
1 
&2-n
&1 
&\cdots&
1   \\
\vdots &\vdots &\vdots&\vdots &\vdots     \\
1&1&1&\cdots&2-n
\end{array}\right) .
\end{eqnarray*}
The remaining determinant can be calculated explicitly by use of the formula
\begin{equation}\label{DetJ1}
\det\left(
\begin{array}{ccccc}
1+x_{1} &1 &1&\cdots&1\\
1&1+x_{2}&1&\cdots&1\\
\vdots&\vdots&   &   &\vdots\\
1&1&1&\cdots&1+x_{n}
\end{array}
\right)
=\prod_{i=1}^{n}x_{i}+\sum_{i=1}^{n}\left(\prod_{j\ne i}x_{j}\right) 
\end{equation}
which yields, with $x_i=1-n$,
\begin{equation}\label{DetJ2}
J
=\frac{(-1)^{n-1}}{n-1}\left(\prod_{k=1}^{n} v_k\right)^{-\frac {n-2}{n-1}} 
= \frac{(-1)^{n-1}}{n-1}\left(\prod_{i=1}^{n} t_i\right)^{-(n-2)} .
\end{equation}
Applying the transformation theorem gives
$$
\cJ(\bm l)
=
\frac{1}{n-1} \underbrace{\int \limits_0^1  \dots 
\int\limits_0^{1}}_{\forall i:\ \prod_1^{n} v_j^{\frac 1{n-1}} \leq v_i} 
\left(1- \a  \sum_{i=1}^{n} v_i   \right)^{N-n}  
\prod_{i=1}^n v_i^{l_{i}- \frac {L}{n-1}}  dv_n \dots dv_1.
$$
In the last step we substitute $v_{i}=\frac 1{\a(N-n)} s_{i}$ and obtain
\begin{eqnarray*}
\left(\frac1 {\a (N-n)}\right)^{n -\frac{L}{n-1}   }
\frac{1}{n-1}
 \underbrace{
 \int \limits_0^{\a (N-n)} \dots 
\int\limits_0^{\a (N-n)}
}_{\forall i:\ \prod_1^{n}  s_j^{\frac 1{n-1}} \leq(\a (N-n))^{\frac{1}{n-1}} s_i} 
\left(1- \frac{\sum_{i=1}^{n} s_i }{N-n}  \right)^{N-n}  
\prod_{i=1}^n s_i^{l_{i}- \frac {L}{n-1}}  ds_n \dots ds_1.
\end{eqnarray*}
\end{proof}

\begin{lemma}\label{le:substitute2}
Let $\alpha>0$, and $\bm l=(l_1, \dots, l_n)$, $L=\sum_1^n l_i $, with  $ l_i <n-1 $ for all $i=1,\dots,n$. Then we have
\begin{eqnarray*}
\cJ (\bm l)
&=& 
\left(\frac 1 {\a (N-n)}\right)^{n - \frac{L}{n-1} }
\frac{1}{n-1}
\sum_{\pi \in \cS_n}
\underbrace{ \int \limits_0^{\a (N-n)}  \int \limits_0^{s_1}  \dots  \int \limits_0^{s_{n-1}}  }_
{\forall i\geq 3 \colon  \left(\frac{ s_{1}\cdots s_{i-1}}{\a (N-n)} \right)^{\frac{1}{i-2}} \leq s_{i} }
\left(1- \frac{\sum_{i=1}^{n} s_{i}}{N-n}  \right)^{N-n}   
\\ && \hskip8cm 
\times \prod_{i=1}^n s_{i}^{l_{\pi(i)} - \frac {L}{n-1} }  
ds_n \dots ds_1 
\end{eqnarray*}
\end{lemma}

\begin{proof}
By the assumption $ l_i <n-1 $ for all $i=1,\dots,n$ the integrals are finite.
The result follows from Lemmata \ref{SetIntegration1} and \ref{le:substitute1}.
\end{proof}

\section{Proof of Lemma \ref{le:main-asymp-exp1}}\label{sec:main-asymp-exp1}

Our goal is to prove Lemma \ref{le:main-asymp-exp1}, which is the asymptotic formula
\begin{eqnarray*}
\cJ(\bm l)
&=&
\int \limits_0^1 \dots  \int\limits_0^{1} 
\left(1- \a \sum_i \prod_{j \neq i} t_j \right)^{N-n} 
\prod_{i=1}^n t_i^{n-2-l_i}  dt_1 \dots dt_n   
\\ &=&
\a^{-n + \frac{L}{n-1} }  (n-1)^{-1} 
\prod_{i=1}^n \Gamma \left( l_{i} -  \frac {L}{n-1} +1 \right) 
\ N ^{-n+\frac{L}{n-1} } 
\left(1+ O\left(N^{-\frac 1{n-2}(\min_k l_k- \frac L{n-1} +1)} \right) \right)
\end{eqnarray*}
as $N \to \infty$, where $n \geq 2$, $0 < \a < \frac{1}{n}$, and $\bm l=(l_1, \dots, l_n)$, $L=\sum_1^n l_i  $,  with $ n-1 > l_{i} > \frac {L}{n-1} - 1$. 

\begin{proof}
By Lemma \ref{le:substitute1} we have
\begin{eqnarray*}
\cJ (\bm l)
&=& 
\left(\frac1 {\a (N-n)}\right)^{ n-\frac{L}{n-1}   }
\frac{1}{n-1}
\int \limits_0^{\a (N-n) } \dots 
\int\limits_0^{\a (N-n) }
\prod_{i=1}^{n}  \I\left( \Big(\a(N-n) \Big)^{-\frac{1}{n-1}}  \prod_{i=1}^{n}  s_j^{\frac 1{n-1}} \leq  s_i \right) 
\\&& \hskip2cm 
 \left(1 -  \frac{\sum_{i=1}^{n} s_i }{N-n} \right)^{N-n}  
\prod_{i=1}^n s_i^{l_{i}- \frac {L}{n-1} }  ds_n \dots ds_1 .
\end{eqnarray*}
Because 
$ e^{t} (1-t) \geq (1+t)(1-t)=(1-t^2) $ for $|t| \leq 1$ and $(1-t^2)^m \geq 1- m t^2$,
we have
\begin{equation}\label{eq:exp-estimate}
0 \leq
e^{-x} - \left(1-\frac x{N-n}\right)^{N-n} 
\leq
e^{-x}\left(1- \left(1-\frac {x^2}{(N-n)^2}\right)^{N-n} \right) 
\leq
e^{-x}\, \frac {x^2}{N-n} 
\end{equation}
for $|x| \leq N-n$.
This yields 
\begin{eqnarray*}
\cJ (\bm l)
&=& 
\left(\frac1 {\a (N-n)}\right)^{ n-\frac{L}{n-1}   }
\frac{1}{n-1}
\int \limits_0^{\a (N-n) } \dots 
\int\limits_0^{\a (N-n) }
\prod_{i=1}^{n}  \I\left( \Big(\a(N-n) \Big)^{-\frac{1}{n-1}}  \prod_{i=1}^{n}  s_j^{\frac 1{n-1}} \leq  s_i \right) 
\\&& \hskip2cm 
e^{- \sum_{i=1}^n s_i} \left(1 +O\left( N^{-1} \sum_{i=1}^n s_i^2 \right) \right)  
\prod_{i=1}^n s_i^{l_{i}- \frac {L}{n-1} }  ds_n \dots ds_1 .
\end{eqnarray*}
Integrating the terms containing $O\left( N^{-1} \sum_{i=1}^n s_i^2 \right)$ yields incomplete Gamma functions times a term $O(N^{- n +\frac{L}{n-1} -1})$. The main term gives
\begin{eqnarray*}
\int \limits_0^{\a (N-n) } \dots 
\int\limits_0^{\a (N-n) }
\lefteqn{
\prod_{i=1}^{n}  \I\left( \Big(\a(N-n) \Big)^{-\frac{1}{n-1}}  \prod_1^{n}  s_j^{\frac 1{n-1}} \leq  s_i \right) 
e^{- \sum_{i=1}^n s_i} 
\prod_{i=1}^n s_i^{l_{i}- \frac {L}{n-1} }  ds_n \dots ds_1 
}&&
\\ &\leq&
\prod_{i=1}^n   \Gamma\left(l_{i}- \frac {L}{n-1} +1 \right) -
\int \limits_{D_N}
e^{- \sum_{i=1}^n s_i} 
\prod_{i=1}^n s_i^{l_{i}- \frac {L}{n-1} }  ds_n \dots ds_1 
\end{eqnarray*}
where 
$D_N$ is the set where at least one of the terms 
$\I\left( \Big(\a(N-n) \Big)^{-\frac{1}{n-1}}  \prod_1^{n}  s_j^{\frac 1{n-1}} \leq  s_i \leq \a (N-n)\right)  $ equals zero.
Thus $D_N$ is covered by the unions of the sets 
$$ 
D_{N,k} = \{s_k \colon s_k \geq \a (N-n) \},\ 
D_{N,k}' = \left\{ s_k \colon s_k^{n-2} \leq \Big(\a(N-n) \Big)^{-1}  \prod_{j \neq k}  s_j  \right\} .
$$
Integration on the set $D_{N,k}$ gives
\begin{eqnarray*}
\int \limits_{D_{N,k}} 
e^{- \sum_{i=1}^n s_i} 
\prod_{i=1}^n s_i^{l_{i}- \frac {L}{n-1} }  ds_n \dots ds_1 
& \leq &
\prod_{i \neq k} \Gamma  \left(l_{i}- \frac {L}{n-1} +1\right)
\int \limits_{\a (N-n)}^\infty 
e^{- s_k} 
s_k^{l_{k}- \frac {L}{n-1} }  ds_k
\\ & = &
O(e^{- \alpha N} N^{l_{i}- \frac {L}{n-1} })
\end{eqnarray*}
and the contribution of the sets $D_{N,k}'$ gives
\begin{eqnarray*}
\int \limits_{D_{N,k}'} 
\lefteqn{
e^{- \sum_{i=1}^n s_i} 
\prod_{i=1}^n s_i^{l_{i}- \frac {L}{n-1} }  ds_n \dots ds_1 
}&&
\\ & \leq &
(\a(N-n))^{-\frac 1{n-2}(l_k- \frac L{n-1} +1)} 
\idotsint \limits_{s_i \geq 0,\, s_k \leq \prod_{j \neq k}  s_j^{\frac 1{n-2}} } 
e^{- \sum_{j \neq k} s_j } 
\prod_{i=1}^n s_i^{l_{i}- \frac {L}{n-1} }  ds_n \dots ds_1 
\\ & = &
O(N^{-\frac 1{n-2}(l_k- \frac L{n-1} +1)}) .
\end{eqnarray*}
Hence the error term of the integration over $D_N$ is of order 
$$
O(N^{-\frac 1{n-2}(\min_k l_k- \frac L{n-1} +1)}) 
$$
for $ l_{i} - \frac {L}{n-1} +1 >0 $.
\end{proof}

\section{Proof of Lemma \ref{le:main-asymp-exp2}}\label{sec:main-asymp-exp2}

Our next goal is to prove Lemma \ref{le:main-asymp-exp2} which deals with the case when some of the $l_i$ are extremal in the sense that $l_i= \frac L{n-1}-1$. 
If for at least three different indices $i, j,k$ we have the strict inequality that $ l_{i}, l_j, l_k >  \frac {L}{n-1} -1 $, 
then we want to prove that
\begin{eqnarray*}
\cJ(\bm l)
&=&
\int \limits_0^1 \dots  \int\limits_0^{1} 
\left(1- \a \sum_i \prod_{j \neq i} t_j \right)^{N-n} 
\prod_{i=1}^n t_i^{n-2-l_i}  dt_1 \dots dt_n   
= 
O\left(  N^{-n+ \frac L{n-1}} (\ln N)^{n-3} \right).
\end{eqnarray*}
If for exactly two different indices $i, j$ we have the strict inequality that $ l_{i}, l_j >  \frac {L}{n-1} -1 $ and equality $l_k = \frac {L}{n-1} -1 $ for all other $l_k$, 
then we will show that
\begin{eqnarray*}
\cJ(\bm l)
&=&
c_n \a^{-n+ \frac L{n-1}} \Gamma(l_i- \frac L{n-1} +1)\Gamma(l_j- \frac L{n-1} +1)  N^{-n+ \frac L{n-1}} (\ln N)^{n-2} (1+O((\ln N)^{-1})) 
\end{eqnarray*}
with some $c_n > 0$. 
First we show that $\cJ(\bm l)$ is at least of order $ N^{-n+ \frac L{n-1}} (\ln N)^{n-2}$, and thus the strict inequality $c_n>0$.
\begin{lemma}\label{OrderMag1}
There is a constant $c_{n,\a}>0$ such that 
for all $n-1 > l_i \geq \frac L{n-1}-1$ we have
$$
\cJ(\bm l)
\geq c_{n,\a} N^{-n+ \frac L{n-1}} (\ln  N)^{\#\{i|l_{i}=\frac L{n-1} -1\}} 
$$
for $N$ sufficiently large. 
\end{lemma}

\begin{proof} 
We use Lemma \ref{le:substitute1}. Since the integrand is positive, for $N$ sufficiently large 
\begin{eqnarray*}
\cJ(\bm l) 
&=& 
c_{n, \a} (N-n)^{-n + \frac L{n-1}} 
\underbrace{
 \int \limits_0^{\a (N-n)} \dots 
\int\limits_0^{\a (N-n)}
}_{\forall i:\ \prod_1^{n}  s_j^{\frac 1{n-1}} \leq(\a (N-n))^{\frac{1}{n-1}} s_i} 
\left(1- \frac{\sum_{i=1}^{n} s_i }{N-n}  \right)^{N-n}  
\prod_{i=1}^n s_i^{l_i - \frac L{n-1} }  ds_n \dots ds_1
\\
&\geq &
c_{n, \a} (N-n)^{-n + \frac L{n-1}} 
\underbrace{ \int \limits_0^{1} \dots  \int\limits_0^{1}
}_{\forall i:\ \prod_1^{n}  s_j^{\frac 1{n-1}} \leq(\a (N-n))^{\frac{1}{n-1}} s_i} 
\left(1- \frac{\sum_{i=1}^{n} s_i }{N-n}  \right)^{N-n}  
\prod_{i=1}^n s_i^{l_i - \frac L{n-1}  }  ds_n \dots ds_1
\\
& \geq &
c_{n, \a} (N-n)^{-n + \frac L{n-1}} 
\underbrace{ \int \limits_0^{1} \dots  \int\limits_0^{1}
}_{\forall i:\ 1 \leq(\a (N-n))^{\frac{1}{n-1}} s_i} 
\left(1- \frac{\sum_{i=1}^{n} s_i }{N-n}  \right)^{N-n}  
\prod_{i=1}^n s_i^{l_i - \frac L{n-1} }  ds_n \dots ds_1
\\
&\geq &
c_{n, \a} (N-n)^{-n + \frac L{n-1}} \left(1- \frac{n}{N-n}  \right)^{N-n}  \!\!\!
 \int \limits_{(\a (N-n))^{-\frac{1}{n-1}}}^{1} \! \dots \!
\int\limits_{(\a (N-n))^{-\frac{1}{n-1}}}^{1} \!
\prod_{i=1}^n s_i^{l_i - \frac L{n-1} }  ds_n \dots ds_1 
\\
&=& c_{n, \a} (N-n)^{-n + \frac L{n-1}} \left(1- \frac{n}{N-n}  \right)^{N-n}  \!\!\!
\,  \,  \,   \prod_{i=1}^n  \int \limits_{(\a (N-n))^{-\frac{1}{n-1}}}^{1} s_i^{l_i - \frac L{n-1} }  ds_i.
\end{eqnarray*}
For those $i$ with $l_i - \frac L{n-1}= -1$, 
$$
\int \limits_{(\a (N-n))^{-\frac{1}{n-1}}}^{1} s_i^{l_i - \frac L{n-1} }  ds_i= \frac{1}{n-1} (\ln \alpha( N-n))
\geq \frac{1}{2(n-1)} \ln N,
$$
and for those $i$ with $l_i - \frac L{n-1}> -1$, 
$$
\int \limits_{(\a (N-n))^{-\frac{1}{n-1}}}^{1} s_i^{l_i - \frac L{n-1} }  ds_i= 
\frac{1}{l_i - \frac{L}{n-1} +1}\left( 1- \alpha( N-n)^{-\frac{l_i - \frac{L}{n-1} +1}{n-1}}\right)
\geq 
\frac{1}{2(l_i - \frac{L}{n-1} +1)},
$$
both for $N$ sufficiently large.
\end{proof}

To show that this yields in fact the correct order we introduce 
in the light of Lemma \ref{le:substitute2} integrals of the type
\begin{eqnarray*}
\cS (\bm q) 
=
\underbrace{ \int \limits_0^{\a (N-n)}  \int \limits_0^{s_1}  \dots  \int \limits_0^{s_{n-1}}  }_
{\forall i\geq 3 \colon  \left(\frac{s_{1}\cdots s_{i-1}}{\a(N-n)} \right)^{\frac{1}{i-2}} \leq s_{i} }
\left(1- \frac{\sum_{i=1}^{n} s_{i}}{N-n}  \right)^{N-n}   
\prod_{i=1}^n s_{i}^{q_i}
\, ds_n \dots ds_1 \ .
\end{eqnarray*}

\begin{lemma}\label{le:part-asymp-exp2}
Assume $\a \leq \frac 1{2n}$, and that $\bm q=(q_1, \dots, q_n) \in \R^n$, $q_i \geq -1$, 
and there are $i \neq j$ with $q_i, q_j >-1$.
Then there is a constant $c_{\bm q,n} \geq 0$ independent of $\a$ such that 
$$
\cS (\bm q)= c_{\bm q,n} (\ln N)^{n-2} +O((\ln N)^{n-3})  
$$
as $N \to \infty$.
More precisely, if $q_1, q_2 > -1$, and $q_3= \dots=q_n=-1$, then 
\begin{equation}\label{eq:asymexp12}
\cS (q_1, q_2, -1, \dots ) + \cS(q_2, q_1, -1, \dots )= 
c_n \Gamma( q_{1}+1) \Gamma(q_{2} +1) (\ln N)^{n-2} +O((\ln N)^{n-3}) 
\end{equation}
with some $c_n \geq 0$.
If there exists an $m\geq 3$ with $q_m > -1$, then $c_{\bm q, n}=0$ and 
\begin{equation}\label{eq:asymexpneq12}
\cS (\bm q)= O((\ln N)^{n-3}) . 
\end{equation}
\end{lemma}
In other words, the only asymptotically contributing terms are those with $q_1,q_2 > -1$ and $q_{3}= \dots = q_{n}=-1. $ We will prove Lemma \ref{le:part-asymp-exp2} below and before this  show that it implies Lemma \ref{le:main-asymp-exp2}.

\begin{proof}[Proof of Lemma \ref{le:main-asymp-exp2}]
For $\bm l=(l_1, \dots, l_n)$, $L=\sum_1^n l_i $, with  $l_i < n-1$, Lemma \ref{le:substitute2} tells us  that
\begin{eqnarray*}
\cJ (\bm l)
&=& 
\left(\frac 1 {\a (N-n)}\right)^{n- \frac{L}{n-1} }
\frac{1}{n-1}
\sum_{\pi \in \cS_n}
\underbrace{ \int \limits_0^{\a (N-n)}  \int \limits_0^{s_1}  \dots  \int \limits_0^{s_{n-1}}  }_
{\forall i\geq 3 \colon  \left(\frac{ s_{1}\cdots s_{i-1}}{\a (N-n)} \right)^{\frac{1}{i-2}} \leq s_{i} }
\left(1- \frac{\sum_{i=1}^{n} s_{i}}{N-n}  \right)^{N-n}   
\\ && \hskip8cm \times
\prod_{i=1}^n s_{i}^{l_{\pi(i)}- \frac {L}{n-1} }  
ds_n \dots ds_1 
\\ &=&
\left(\frac 1 {\a (N-n)}\right)^{n- \frac{L}{n-1} }
\frac{1}{n-1}
\sum_{\pi \in \cS_n} \cS (\bm l_\pi - (\frac {L}{n-1}) \bm 1) .
\end{eqnarray*}
Assume that 
$ l_{ i} \geq \frac L{n-1} -1$ for all $i$, and there exists some tuple $ i \neq j$ with $ l_{i} ,  l_{j} > \frac L{n-1}  -1$. 
If $ l_{\pi(1)} ,  l_{\pi(2)} > \frac L{n-1}  -1$ and $ l_{\pi(i)} = \frac L{n-1} -1$ for all $i \geq 3$, we have that 
$$
\cS (\bm l_\pi - (\frac {L}{n-1}) \bm 1) = c_{\bm l_\pi - (\frac {L}{n-1}) \bm 1,n} (\ln N)^{n-2} +O((\ln N)^{n-3})
$$
where the constant is non-negative.
If $ l_{\pi(i)} > \frac L{n-1} -1$ for some $i \geq 3$, then
$$
\cS (\bm l_\pi - (\frac {L}{n-1}) \bm 1) = O( (\ln N)^{n-3} ) .
$$
Hence, depending on $\bm l=(l_1, \dots, l_n)$, there are two cases.
\begin{itemize}
 \item We have $ l_{k} = \frac L{n-1} -1$ for all except two indices $i \neq j$: Then there are $(n-2)!$ permutations which bring $l_i, l_j$ into the first two places with order $(l_i, l_j)$, resp. $(l_j, l_i)$ and allow for an application of \eqref{eq:asymexp12}. 
 All other permutations add terms of order $O( (\ln N)^{n-3})$.
 Summing over these possibilities, we have 
 \begin{eqnarray*}
\cJ (\bm l)
&=& 
\left(\frac 1 {\a (N-n)}\right)^{n- \frac{L}{n-1} }
\frac{(n-2)!}{n-1}
 c_n \Gamma(l_i- \frac L{n-1} +1)\Gamma(l_j- \frac L{n-1} +1)  
\\ && \hskip8cm \times 
(\ln N)^{n-2} (1+O((\ln N)^{-1}) 
\\ &=&
c_n \a^{-n+ \frac L{n-1}} \Gamma(l_i- \frac L{n-1} +1)\Gamma(l_j- \frac L{n-1} +1)  N^{-n+ \frac L{n-1}} (\ln N)^{n-2}  (1+O((\ln N)^{-1});
\end{eqnarray*}
 \item There exist at least three different $ l_i, l_j, l_k > \frac L{n-1} -1$. This yields
$$
\cJ(\bm l) = O\left(  N^{-n+ \frac L{n-1}} (\ln N)^{n-3} \right) .
$$
\end{itemize}
The implicit constants in $O(\cdot )$ may depend on $\a$.
These estimates imply Lemma \ref{le:main-asymp-exp2}. 
\end{proof}

\begin{proof}[Proof of Lemma \ref{le:part-asymp-exp2}]
The proof of the lemma is divided into four parts. 
Lemma \ref{le:crit-int} and Lemma \ref{le:non-crit-int} give the crucial estimates.
Equation \eqref{eq:asymexp12} when $q_3= \dots = q_n=-1$ follows from Lemma \ref{le:domterm},
$$
\cS (q_1, q_2, -1, \dots, -1) 
=
c_n \ln(N-n)^{n-2} \int \limits_0^{\a (N-n)}  
\int\limits_0^{s_1} 
\left(1- \frac{s_1+s_2}{N-n}  \right)^{N-n} 
s_1^{q_{1} }s_2^{q_{2} } ds_2 ds_1 
\ + O((\ln N)^{n-3}) .
$$
We replace $\left(1- \frac{s_1+s_2}{N-n}  \right)^{N-n}$ by the exponential function using \eqref{eq:exp-estimate}:
\begin{align*}
\int \limits_0^{\a (N-n) }  
\int\limits_0^{s_1} 
\left(1- \frac{s_1+s_2}{N-n}  \right)^{N-n} 
s_1^{q_{1} }s_2^{q_{2} } ds_2 ds_1 
=&
\int \limits_0^{\a (N-n) }  
\int\limits_0^{s_1} 
e^{-(s_1+s_2)} \left(1+ O(N^{-1} (s_1+s_2)^2) \right) 
s_1^{q_{1} }s_2^{q_{2} } ds_2 ds_1 
\\ =&
\int \limits_0^{\infty}  
\int\limits_0^{\infty} \I(s_2 \leq s_1) 
e^{- (s_1+s_2)} 
s_1^{q_{1} }s_2^{q_{2} } ds_2 ds_1 
\\ &- 
\int \limits_{\a (N-n) }^\infty  
\int\limits_0^{s_1} 
e^{-(s_1+s_2)} 
s_1^{q_{1} }s_2^{q_{2} } ds_2 ds_1 
+ O(N^{-1})   .
\end{align*}
Clearly the integral in the last line is of order 
$O(e^{-N} N^{q_1 })$.
Hence 
\begin{eqnarray*}
\left( \cS (q_1, q_2, -1, \dots, -1) + \cS (q_2, q_1, -1, \dots, -1) \right)
&=&
c_n \Gamma( q_{1}+1) \Gamma(q_{2} +1) (\ln N)^{n-2} +O ((\ln N)^{n-3}).
\end{eqnarray*}

Equation \eqref{eq:asymexpneq12} is proved in Lemma \ref{le:nondomterm}.
\end{proof}

\begin{lemma}\label{le:crit-int}
Assume $s_n \leq \dots \leq s_1 \leq \a (N-n) $, $3 \leq m \leq n$, and $s_{m-1} \leq 1$. Then for $\a \leq \frac 1{2n}$ and $k \geq 0$ we have 
\begin{align*}
&\frac 1{k+1} \left(1- \frac{\sum_1^{m-1} s_i}{N-n}\right)^{N-n} 
\Bigg( \left(-\frac{1}{m-2} \ln\left( \frac{s_{1}\cdots s_{{m-1}}}{\a (N-n)} \right) \right)^{k+1} -  
(-\ln s_{{m-1}})^{k+1} 
- 2 \Gamma (k+2)\Bigg)
\\ &
\leq 
\int\limits_{\left(\frac{s_{1}\cdots s_{{m-1}}}{\a (N-n)}
\right)^{\frac{1}{m-2}}}^{s_{{m-1}}} 
\left(1- \frac{\sum_1^{m} s_i}{N-n}  \right)^{N-n}   
s_{m}^{-1} (- \ln s_{m})^k ds_{m}
\leq 
\\ &
\frac 1{k+1} \left(1- \frac{\sum_1^{m-1} s_i}{N-n}\right)^{N-n} 
\Bigg( \left(-\frac{1}{m-2} \ln\left( \frac{s_{1}\cdots s_{{m-1}}}{\a (N-n)} \right) \right)^{k+1} -  
(-\ln s_{{m-1}})^{k+1}   \Bigg) .
\end{align*}
\end{lemma}

\begin{proof}
We use the notation $S:= \frac{\sum_1^{{m-1}} s_i}{N-n}$. By assumption $\alpha \leq \frac 1{2n}$. This implies
$$S= \frac{\sum_1^{m-1} s_i}{N-n} \leq 
\frac{ns_1}{N-n} \leq n \a \leq \frac 12  . $$ 
And  for $S \leq \frac 12$ and $x\geq 0$ we have 
\begin{equation}\label{eq:1-absch}
(1-S)(1-2x)  
\leq \left( 1- (S+ x)\right) \leq (1- S) .
\end{equation}
The essential observation is that for $a,b \in (0,1)$ and $k \geq 0$
\begin{equation}\label{eq:ln-gamma}
\int_a^b (-\ln s)^k ds 
= \int_{- \ln b}^{- \ln a} t^k e^{-t} dt  
\leq \int_0^\infty t^k e^{-t} dt = \Gamma(k+1) 
\end{equation}
and
\begin{equation}\label{eq:s-1ln-gamma}
\int_a^b s^{-1} (-\ln s)^k ds 
= 
-\frac 1{k+1} (- \ln s)^{k+1} \Big\vert_a^b 
= 
\frac 1{k+1} (- \ln b)^{k+1}   
-\frac 1{k+1} (- \ln a)^{k+1}  .
\end{equation}
Because of \eqref{eq:1-absch} and \eqref{eq:s-1ln-gamma} we obtain
\begin{eqnarray*}
\int\limits_{\left(\frac{ s_{1}\cdots s_{{m-1}}}{\a (N-n)}
\right)^{\frac{1}{m-2}}}^{s_{{m-1}}} \lefteqn{
\left(1- S - \frac{ s_{m}}{N-n}  \right)^{N-n}   
s_{m}^{-1} (-\ln s_{m})^k ds_{m} 
} &&
\\ &\leq&
\left(1- S\right)^{N-n} 
\int\limits_{\left(\frac{ s_{1}\cdots s_{{m-1}}}{\a (N-n)}
\right)^{\frac{1}{m-2}}}^{s_{{m-1}}}
s_{m}^{-1} (-\ln s_{m})^k ds_{m}
\\ &=&
\frac 1{k+1} \left(1- S\right)^{N-n} 
\left( \left(-\frac{1}{m-2} \ln\left( \frac{ s_{1}\cdots s_{{m-1}}}{\a (N-n)} \right) \right)^{k+1} -  
(-\ln s_{{m-1}})^{k+1}  
\right).
\end{eqnarray*}
Again by \eqref{eq:1-absch} and \eqref{eq:s-1ln-gamma}, by the elementary inequality
$ (1-y)^{k}  \geq  (1-ky) $ for $y\leq 1$ and by \eqref{eq:ln-gamma}
\begin{eqnarray*}
\int\limits_{\left(\frac{ s_{1}\cdots s_{{m-1}}}{\a (N-n)}
\right)^{\frac{1}{m-2}}}^{s_{{m-1}}} \lefteqn{
\left(1- S - \frac{ s_{m}}{N-n}  \right)^{N-n}   
s_{m}^{-1} (- \ln s_{m})^k ds_{m} 
} &&
\\ &\geq&
\left(1- S\right)^{N-n} 
\int\limits_{\left(\frac{ s_{1}\cdots s_{{m-1}}}{\a (N-n)}
\right)^{\frac{1}{m-2}}}^{s_{{m-1}}}
(1-2  s_{m} ) s_{m}^{-1} (-\ln s_{m})^k ds_{m}
\\ &=&
\frac 1{k+1} \left(1- S\right)^{N-n} 
\Bigg( \left(-\frac{1}{m-2} \ln\left( \frac{ s_{1}\cdots s_{{m-1}}}{\a (N-n)} \right) \right)^{k+1} -  
\\ && \hskip5cm 
-\ (-\ln s_{{m-1}})^{k+1} 
-2(k+1) \Gamma (k+1)\Bigg) .
\end{eqnarray*}
This proves the lemma.
\end{proof}

With the help of this lemma we determine the asymptotic behavior of the dominant terms.

\begin{lemma}\label{le:domterm}
There is a constant $c_n$, such that for $q_1, q_2 > -1 $ and $\a \leq \frac 1{2n}$ we have
$$
\cS (q_1, q_2, -1, \dots, -1) 
=
c_n \left(\ln(N-n)\right)^{n-2} \int \limits_0^{\a (N-n)}  
\int\limits_0^{s_1} 
\left(1- \frac{s_1+s_2}{N-n}  \right)^{N-n} 
s_1^{q_{1} }s_2^{q_{2} } ds_2 ds_1 
\ + O((\ln N)^{n-3})
$$
\end{lemma}
\begin{proof}
We denote the range of integration of $\cS (\bm q)$ by $I$ and dissect this along the sets 
$$I_k :=\{0 \leq s_n \leq \dots \leq s_k \leq 1 \leq s_{k-1} \leq \dots  \leq s_3  \} , $$
for $k=3, \dots, n$, and 
$$I_{n+1} :=\{1 \leq s_n \leq \dots \leq s_3  \} .$$
The dominant term is the one with $I \cap I_{3}=I \cap \{0 \leq s_n \leq \ldots \leq s_3 \leq 1\}$  as range of 
integration. Hence in the first part of the proof we assume $s_i \leq 1$ for $i=3, \dots, n$.
For $m=2, \dots, n-1 $ we define 
$$
\cS _{n-m}(s_1, \dots, s_m)= 
\int\limits_{\left(\frac{ s_{1}\cdots s_{m}}{\a (N-n)}
\right)^{\frac{1}{m-1}}}^{s_{m}}
\cdots
\int\limits_{\left(\frac{s_{1}\cdots s_{n-1}}{\a (N-n)}\right)^{\frac{1}{n-2}}}^{s_{n-1}}
\left(1- \frac{\sum_{i=1}^{n} s_{i}}{N-n}  \right)^{N-n}   
\prod_{i=m+1}^n s_i^{-1} ds_n \dots ds_{m+1}
$$
and claim that 
\begin{eqnarray} \label{eq:ind-Jm}
\cS _{n-m}(s_1, \dots, s_m)
&=& 
\left(1- \frac{\sum_{i=1}^{m} s_{i}}{N-n}  \right)^{N-n}   
( P_{n-m} (\ln (N-n), \ln s_1, \dots, \ln s_m)  
\\ && \hskip3cm \nonumber +
E_{n-m} (\ln (N-n), \ln s_1, \dots, \ln s_m)) 
\end{eqnarray}
where $P_{n-m}$ is a homogeneous polynomial of degree $n-m$ independent of $\a$, and the error term $E_{n-m}$ is a 
function whose absolute value is bounded by a polynomial $Q_{n-m-1}$ of degree at most $n-m-1$ whose coefficients may depend on $\a$.
To shorten the following formulae we suppress the arguments of $P_{n-m}, \, E_{n-m}$ and 
$Q_{n-m-1}$ from now on. 

We use induction in $m$, starting with $m=n-1$ and going down to $m=2$.
For $m=n-1$ and $ P_0=1$ in the first step we obtain $\cS _{1}=(1- (N-n)^{-1}(\sum s_i) ) (P_1 +E_1)$ by Lemma 
\ref{le:crit-int} (where $k=0$)
with 
$$-\frac 1{n-2}\ln\left(\frac{1 }{\a}\right) -2 =-Q_0  \leq E_1 \leq - \frac 1{n-2} \ln\left(\frac{1 }{\a}\right) $$
and
$$
P_1= 
-\frac{1}{(n-2)} \ln\left( \frac{s_{1}\cdots s_{{n-1}}}{N-n} \right) 
+ \ln s_{{n-1}}
.$$
Assume that \eqref{eq:ind-Jm} holds. Then
\begin{eqnarray*}
\cS _{n-m+1} (s_1, \dots, s_{m-1})
&=& 
\int\limits_{\left(\frac{ s_{1}\cdots s_{m-1}}{\a (N-n)}
\right)^{\frac{1}{m-2}}}^{s_{m-1}}
\left(1- \frac{\sum_{i=1}^{m} s_{i}}{N-n}  \right)^{N-n} 
(P_{n-m} + E_{n-m})   s_m^{-1} 
ds_m
\end{eqnarray*}
with 
\begin{equation}\label{eq:defp}
P_{n-m} = 
\sum_{k=0}^{n-m} (-\ln s_m)^k   p_{n-m-k}  
\end{equation}
where the coefficients $p_{n-m-k}$ are polynomials in $\ln(N-n), \ln s_1, \dots, \ln s_{m-1} $ of 
degree $(n-m-k)$ independent of $\a$. And the absolute value of $E_{n-m}$ is bounded by a polynomial $Q_{n-m-1}$ of 
degree $n-m-1$.

In Lemma \ref{le:crit-int}  both bounds are - up to the term $(1-S)^{N-n}$ - polynomials of degree 
$k+1$ where the  sum of the monomials of top degree $k+1$ is denoted by $H_{k+1}$ and is independent of $\a$.  We have
\begin{align*}
H_{k+1} =
\frac 1{k+1} 
\Bigg( \left(\frac{-1}{m-2}\right)^{k+1}  \left( \ln\left( \frac{s_{1}\cdots s_{{m-1}}}{N-n} 
\right) \right)^{k+1} -  
(-\ln s_{{m-1}})^{k+1}   \Bigg).
\end{align*}
Thus by Lemma \ref{le:crit-int} the integration of $P_{n-m}$
yields homogeneous polynomials $H_{k+1}$ of 
degree $k+1$, 
and hence a homogeneous polynomial $P_{n-m+1}$ of 
degree $n-m+1$ 
\begin{align*}
P_{n-m+1}  = 
\sum_{k=0}^{n-m} H_{k+1}  p_{n-m-k}
\end{align*}
independent of $\a$.
The other terms of lower degree and the error term in Lemma \ref{le:crit-int} produce 
error terms which can be bounded by a polynomial of degree $k$. Multiplied by the polynomials 
$p_{n-m-k}$ from the representation \eqref{eq:defp} this yields an error term $E'_{n-m+1}$ bounded 
by a polynomial $Q'_{n-m}$ in $\ln(N-n), \ln s_1, \dots, \ln s_{m-1}$ of order $n-m$,
$$ 
|E'_{n-m+1} | \leq 
Q'_{n-m}  .  
$$ 
For the absolute value of the integration over $E_{n-m}$ we obtain
\begin{eqnarray*}
| E''_{n-m+1}  |
&=&
\left| 
\int\limits_{\left(\frac{s_{1}\cdots s_{m-1}}{\a (N-n)}
\right)^{\frac{1}{m-3}}}^{s_{m-1}}
\left(1- \frac{\sum_{i=1}^{m} s_{i}}{N-n}  \right)^{N-n} 
E_{n-m}   s_m^{-1} 
ds_m
\right|
\\ &\leq&
\int\limits_{\left(\frac{ s_{1}\cdots s_{m-1}}{\a (N-n)}
\right)^{\frac{1}{m-3}}}^{s_{m-1}}
\left(1- \frac{\sum_{i=1}^{m} s_{i}}{N-n}  \right)^{N-n} 
|E_{n-m} | s_m^{-1} 
ds_m
\\ &\leq&
\int\limits_{\left(\frac{ s_{1}\cdots s_{m-1}}{\a (N-n)}
\right)^{\frac{1}{m-3}}}^{s_{m-1}}
\left(1- \frac{\sum_{i=1}^{m} s_{i}}{N-n}  \right)^{N-n} 
Q_{n-m-1}   s_m^{-1} 
ds_m 
\leq 
Q''_{n-m}  
\end{eqnarray*}
where in the third line we used Lemma \ref{le:crit-int} again which leads to a polynomial  
$Q''_{n-m}$ of degree $n-m$.
Hence $E_{n-m+1} := E'_{n-m+1}+E''_{n-m+1}$ is bounded by $Q'_{n-m} + Q''_{n-m}$, a 
polynomial of degree $n-m$. This proves \eqref{eq:ind-Jm}.
On $I \cap I_3$ 
we take $\min(s_2, 1)$ as the upper 
limit of integration with respect to $s_3$.

Thus we obtain on $I \cap I_3$ 
\begin{eqnarray*}
\cS _{n-2}(s_1, s_2)
&=& 
\left(1- \frac{s_1+s_2}{N-n}  \right)^{N-n} 
\\ &&
 \times \Big(P_{n-2} (\ln(N-n), \ln s_1, \ln (\min(s_2,1)))  + E_{n-2} (\ln(N-n), \ln s_1, \ln (\min(s_2,1))) \Big) .
\end{eqnarray*}

\bigskip
It remains to consider the last two integrations with $q_{1},q_{2} > -1$. The dominating term in 
Lemma \ref{le:domterm} is the term of $P_{n-2}$  with $(\ln (N-n))^{n-2}$,
\begin{equation}\label{eq:domterm}
(\ln(N-n))^{n-2} \int \limits_0^{\a (N-n)}  
\int\limits_0^{s_1} 
\left(1- \frac{s_1+s_2}{N-n}  \right)^{N-n} 
s_1^{q_{1} }s_2^{q_{2} } ds_2 ds_1 .
\end{equation}
For the terms $(\ln (N-n))^k  (\ln s_1)^{j_1} (\ln (\min(s_2,1)))^{j_2}  $  with $k= 
n-2-j_1-j_2<n-2$ 
we obtain 
\begin{eqnarray}
\int \limits_0^{\a (N-n)} 
\lefteqn{
\int\limits_0^{s_1} 
\left(1- \frac{s_1+s_2}{N-n}  \right)^{N-n} 
\left| (\ln (N-n))^k  (\ln s_1)^{j_1} (\ln (\min(s_2,1)))^{j_2}  
s_1^{q_{1} }s_2^{q_{2} } \right| 
 ds_2 ds_1 
 }&& \nonumber
\\ &\leq& \nonumber
(\ln (N-n))^k
\int \limits_0^{\infty}  
\int\limits_0^{\infty} 
e^{-s_1-s_2} 
|\ln s_1|^{j_1} |\ln s_2|^{j_2}  
s_1^{q_{1} }s_2^{q_{2} } 
 ds_2 ds_1 
\\&=& \label{eq:domterm2}
O((\ln (N-n))^k)
\end{eqnarray}
with $k \leq n-3$, since integrals of the form $\int_0^{\infty} e^{-t} t^k |\ln t|^j dt$ are 
convergent.

For the integral over the error term we get
\begin{eqnarray*}
\Big|  \int \limits_0^{\a (N-n)}  
\int\limits_0^{s_1} 
\lefteqn{
\left(1- \frac{s_1+s_2}{N-n}  \right)^{N-n} 
E_{n-2} (\ln(N-n), \ln s_1, \ln s_2)
 s_1^{q_{1} }s_2^{q_{2} } ds_2 ds_1 \Big|
}&&
\\ &\leq&
 \int \limits_0^{\a (N-n)}  
\int\limits_0^{s_1} 
\left(1- \frac{s_1+s_2}{N-n}  \right)^{N-n} 
Q_{n-3} (\ln(N-n), \ln s_1, \ln (\min(s_2,1)))  
 s_1^{q_{1} }s_2^{q_{2} } ds_2 ds_1
\\ &=&
O((\ln(N-n))^{n-3}) .
\end{eqnarray*}
Combining these estimates yields Lemma \ref{le:domterm} for $s_3 \leq 1$, i.e. on $I \cap I_3$. 

\medskip
It remains to show that the integration over $I_4\cup\dots \cup I_{n+1}$ is of order $O((\ln 
N)^{n-3})$.
Consider the range of integration $I \cap I_k$, $k \geq 4$ with
$$I_k :=\{0 \leq s_n \leq \dots \leq s_k \leq 1 \leq s_{k-1} \leq \dots  \leq s_3  \} . $$
Then the integrations up to $s_k$ just yield \eqref{eq:ind-Jm}
and in the remaining integrations we have 
\begin{eqnarray}\label{eq:rest-int}
\Big| \int\limits_{I \cap I_k} 
\lefteqn{
\Big(1 -  \frac{\sum_{i=1}^{k-1} s_i}{N-n}  \Big)^{N-n}  
\prod_{i=1}^{k-1} s_i^{-1 } 
\cS _{n-k+1}(s_1, \dots, s_{k-1})
ds_{k-1} \dots ds_1 \Big| 
} &&
\\ &\leq& \nonumber
\int\limits_1^{\infty}
\cdots
\int\limits_1^{\infty}
 \exp\Big\{ -  \sum_{i=1}^{k-1} s_i  \Big\}  
| \cS _{n-k+1}(s_1, \dots, s_{k-1}) |
ds_{k-1} \dots ds_1
\\ & = &  \nonumber
O\left( (\ln(N-n))^{n-k+1} \right) =
O\left( (\ln(N-n))^{n-3} \right)
\end{eqnarray}
since $\cS _{n-k+1}$ is bounded by polynomials in $\ln (N-n), \ln(s_1), \dots, \ln s_{k-1} $ 
of order $n-k+1$, and all occurring integrals 
$$
\int\limits_1^{\infty}
\cdots
\int\limits_1^{\infty}
 \exp\Big\{ -  \sum_{i=1}^{k-1} s_i  \Big\} 
(\ln s_1)^{j_1} \dots (\ln s_{k-1})^{j_{k-1}}
ds_{k-1} \dots ds_1
$$
are 
finite.
This finishes the proof of lemma \ref{le:domterm}.
\end{proof}

\bigskip
For the second part of Lemma \ref{le:part-asymp-exp2}, i.e. for equation (\ref{eq:asymexpneq12})  we investigate the terms with $ q_{m} > -1 $ for 
some $m \in \{3, \dots,n\} $.
We start by restating the following simple analogue of Lemma \ref{le:crit-int}. We recall that $S=\frac{\sum_{i=1}^{m-1} s_i}{N-n}$.

\begin{lemma}\label{le:non-crit-int}
For $q _m> -1$, $k \geq 0$, $s_{{m-1}} \leq 1$ and $s_{m-1} \leq s_2$ we have
$$
\int\limits_{\left(\frac{s_{1}\cdots s_{{m-1}}}{\a (N-n)}
\right)^{\frac{1}{m-2}}}^{s_{{m-1}}} 
\left(1- S - \frac{ s_{m}}{N-n}  \right)^{N-n}   
s_{m}^{q_m} (- \ln s_{m})^k ds_{m}
\leq
c_{k,q_m}
\left(1- S\right)^{N-n} s_{2}^{q_m+1}
(- \ln s_{m-1})^k .
$$
\end{lemma}
\begin{proof}
We use that the antiderivative of $e^{-t} t^k$ is given by $e^{-t} P_k(t)$ where $P_k$ is a polynomial of degree $k$.
 \begin{eqnarray*}
\int\limits_{\left(\frac{s_{1}\cdots s_{{m-1}}}{\a (N-n)}
\right)^{\frac{1}{m-2}}}^{s_{{m-1}}} 
\lefteqn{
\left(1- S - \frac{ s_{m}}{N-n}  \right)^{N-n}   
s_{m}^{q_m} (- \ln s_{m})^k ds_{m}
}&&
\\ &\leq & 
\left(1- S \right)^{N-n}   
\int\limits_{- \ln s_{{m-1}}}^{\infty} e^{-t(q_m+1)} t^k dt
\\ &  &  =
(q_m+1)^{-(k+1)} \left(1- S \right)^{N-n}   
\int\limits_{- (q_m+1) \ln s_{{m-1}}}^{\infty} e^{-t} t^k dt
\\ &  &  =
(q_m+1)^{-(k+1)} \left(1- S \right)^{N-n}   
e^{-t} P_{k}(t) \vert_{- (q_m+1) \ln s_{m-1}}^{\infty} 
\\ & \leq & 
c_{k,q_m} \left(1- S\right)^{N-n} s_{m-1}^{q_m+1}
(- \ln s_{m-1})^k
\\ & \leq & 
c_{k,q_m} \left(1- S\right)^{N-n} s_{2}^{q_m+1}
(- \ln s_{m-1})^k .
\end{eqnarray*}
\end{proof}

\begin{lemma}\label{le:nondomterm}
Assume $q_n, \dots, q_{m+1}= -1 $, $q_m >-1$ for some $m \geq 3$, and $q_{m-1}, \dots, q_1 \geq -1$. Then we have
$$
\cS (q_1, \dots, q_n) 
=
O((\ln(N-n))^{n-3}).
$$
\end{lemma}
\begin{proof}
We proceed precisely as in the previous proof of Lemma \ref{le:domterm}.
We denote the range of integration by $I$ and dissect this set by
$$I_k :=\{0 \leq s_n \leq \dots \leq s_k \leq 1 \leq s_{k-1} \leq \dots  \leq s_3  \} , $$
for $k=3, \dots, n+1$.

First we deal with the term with $I \cap I_{3}=I \cap \{\ldots s_3 \leq 1\}$  as range of 
integration, hence we assume $s_i \leq 1$ for $i=3, \dots, n$.
We define 
$$
\cS _{n-m}(s_1, \dots, s_m)= 
\int\limits_{\left(\frac{ s_{1}\cdots s_{m}}{\a (N-n)}
\right)^{\frac{1}{m-1}}}^{s_{m}}
\cdots
\int\limits_{\left(\frac{ s_{1}\cdots s_{n-1}}{\a (N-n)}\right)^{\frac{1}{n-2}}}^{s_{n-1}}
\Big(1- \frac{ \sum_{i=1}^{n} s_i}{N-n}  \Big)^{N-n} 
\prod_{i=m+1}^n s_i^{-1} ds_n \dots ds_{m+1}.
$$
We know from the proof of Lemma \ref{le:crit-int}  that 
\begin{equation*} 
| \cS _{n-m}(s_1, \dots, s_m) | \leq 
P_{n-m} (\ln(N-n), \ln s_1, \dots, \ln s_m)  .
\end{equation*}
Because $q_{m}  > -1$, the next 
integration by Lemma \ref{le:non-crit-int} yields as a bound a polynomial of again degree 
$n-m$ in $\ln(N-n), \ln s_1, \dots, \ln s_{m-1}$ times $s_2^{q_m+1}$.

Proceeding in this way, each integration with respect to $s_i$ with $q_i=-1$ increases the degree of the polynomial bound by one, 
and each integration with respect to $s_m$ with $q_m >-1$ leads to a polynomial bound 
again of the same degree and multiplies this new polynomial bound by $s_2^{q_m+1}$.

Thus we obtain on $I \cap I_3$ 
$$
\cS _{n-2}(s_1, s_2)= 
P_{q_{-}} (\ln(N-n), \ln s_1, \ln (\min(s_2,1)))  s_2^{q_{+}}
$$
where we put
$q_{-}=\sum_{l=3}^n  \I(q_l=-1) $ and 
$q_{+}=\sum_{l=3}^n (q_l+1) \I(q_l > -1)$.
This now yields 
$$
\int \limits_0^{\a (N-n)}  
\int\limits_0^{s_1} 
\Big( 1 - \frac{s_1+s_2}{N-n} \Big)^{N-n}  
P_{q_{-}} (\ln(N-n), \ln s_1, \ln (\min(s_2,1)))  
 s_1^{q_{1} }s_2^{q_{2} + q_{+} } ds_2 ds_1 
$$
as a bound for $\cS (q_1, \dots, q_n)$. By our assumption 
$q_{-} \leq n-3$,
$q_{+} > 0$, thus 
$$ q_2 + q_+ >-1,\  
\mbox{ and }\ 
q_1+q_2 + q_{+} +1 >-1. $$
Hence  $\cS (q_1, \dots, q_n)$ is bounded by 
\begin{eqnarray*}
\int \limits_0^{\a (N-n)}  
\int\limits_0^{s_1} 
\lefteqn{
e^{- (s_1+s_2)}  
P_{q_{-}} (\ln(N-n), \ln s_1, \ln (\min(s_2,1)))  
 s_1^{q_{1} }s_2^{q_{2} + q_{+} } ds_2 ds_1 
}&&
\\ & \leq &
\frac 1{q_{2} + q_{+}  +1}
\int \limits_0^{\infty}  
e^{- s_1}  
P_{q_{-}} (\ln(N-n), \ln s_1, \ln (\min(s_1,1)))  
 s_1^{q_{1}+q_{2} + q_{+} +1}ds_1 
\\&=&
O((\ln(N-n))^{q_{-}})
\\ &=&
O((\ln(N-n))^{n-3})
\end{eqnarray*}
on $I \cap I_3$. 
On $I \cap I_k$ with $k \geq 4$, the term $\cS(q_1, \dots, q_m)$ is by monotonicity (observe that $s_k, \dots, s_n \leq 1$) bounded by $\cS(q_1, \dots, q_{k-1}, -1, \dots, -1)$  which in turn is bounded by 
$$
\int\limits_{I \cap I_k} \exp\left\{ - \sum_{i=1}^{k-1} s_i\right\} \left( \prod\limits_{i=1}^{k-1} s_i^{q_i} \right)| P_{n-k-1}+E_{n-k-1} | ds_{k-1} \dots ds_1
\leq O\left((\ln (N-n))^{n-3} \right)
$$
as in the proof of Lemma \ref{le:domterm}, Equation \eqref{eq:rest-int}. 
Hence the integration over $I_4\cup I_5 \dots$ leads to a term of order $O((\ln(N-n))^{n-3})$. 
This proves our lemma.
\end{proof}

\end{appendix}

\subsection*{Acknowledgement}
The authors want to thank the referees for their careful reading of the manuscript and various suggestions for improvement.


\vskip 4mm

Matthias Reitzner \\
  {\small        Universit\"at Osnabr\"uck}\\
    {\small      Fachbereich 6: Mathematik/Informatik}\\
   {\small       Albrechtstr. 28 a }\\
    {\small      49076 Osnabr\"uck, Germany} \\
          {\small \tt matthias.reitzner@uni-osnabrueck.de} \\

   Carsten Sch\"utt \\
     {\small       Christian Albrechts Universit\"at \hskip 20mm   Case Western Reserve University}\\
        {\small    Mathematisches Seminar \hskip 29mm   Department of Mathematics}\\
        {\small    24098 Kiel, Germany\hskip 35mmCleveland, Ohio 44106, U. S. A.} \\
         {\small \tt schuett@math.uni-kiel.de}   \\

     \and Elisabeth Werner\\
{\small Department of Mathematics \ \ \ \ \ \ \ \ \ \ \ \ \ \ \ \ \ \ \ Universit\'{e} de Lille 1}\\
{\small Case Western Reserve University \ \ \ \ \ \ \ \ \ \ \ \ \ UFR de Math\'{e}matiques }\\
{\small Cleveland, Ohio 44106, U. S. A. \ \ \ \ \ \ \ \ \ \ \ \ \ \ \ 59655 Villeneuve d'Ascq, France}\\
{\small \tt elisabeth.werner@case.edu}\\ \\

\end{document}